\documentclass{amsart}
\usepackage{array}
\usepackage{graphicx,verbatim, amsmath, amssymb, amsthm, amsfonts, epsfig, amsxtra,ifthen,mathtools,epstopdf,caption,enumerate,hhline,bbm,capt-of,longtable} 
\usepackage{tikz}
\usepackage{subcaption}                 %%% added 
\usepackage[mode=image|tex]{standalone} %%% added
\usepackage{tikz}
\usetikzlibrary{arrows,backgrounds,shapes}
\tikzset{point/.style = {circle,draw=black, fill=black, inner sep=0pt,minimum size=4pt}}
\tikzset{arc/.style = {black, thick, rounded corners=10pt}}
\tikzset{barc/.style = {black, thick, rounded corners=6pt}}

\tikzset{cross/.style={cross out, draw=black, very thick, minimum size=6pt, inner sep=0pt, outer sep=0pt}}

\tikzset{none/.style = {red, dotted, very thick, rounded corners=10pt}}
\tikzset{ntwo/.style = {blue, dashed,thick, rounded corners=10pt}}
\tikzset{end/.style = {circle,draw=purple!30, fill=purple, opacity=.3, inner sep=0pt,minimum size=10pt}}
\tikzset{bone/.style = {red, dotted, very thick, rounded corners=6pt}}
\tikzset{btwo/.style = {blue, dashed,thick, rounded corners=6pt}}

\tikzset{background rectangle/.style={draw=black, 
    line width=2pt}} 
\tikzset{extpt/.style = {circle,draw=black, fill=black, inner sep=0pt,minimum size=4pt}}
\tikzset{intpt/.style = {circle,draw=black,inner sep=3.5pt,minimum size=4.5pt}}
\tikzset{whitept/.style= {circle,draw=white,fill=white,minimum size=4.5pt}}

\tikzset{cover/.style = {thick}}

\tikzset{forgetnode/.style = {opacity=.4}}
\tikzset{forgetcover/.style = {gray, opacity=.3}}
\tikzset{keepcover/.style = {black, thick}}
\tikzset{eclass/.style = {line width=15pt,blue,line cap=round,opacity=.3}}
\tikzset{ept/.style = {circle,draw=blue, fill=blue,inner sep=0pt,minimum size=15pt,opacity=.3}}

\tikzset{dot/.style = {circle,draw=black, fill=black, inner sep=0pt,minimum size=5pt}}

\tikzset{barrier/.style= {red, very thick}}  
\usepackage[bookmarks=true, bookmarksopen=false,%
    colorlinks=true,%
    linkcolor=darkblue,%
    citecolor=darkblue,%
    filecolor=darkblue,%
    menucolor=darkblue,%
    linktoc=all,%
    urlcolor=darkblue
]{hyperref}
\usepackage[capitalize,noabbrev]{cleveref}
\DeclareGraphicsExtensions{.ps}
\graphicspath{ {goodpictures/} }
\usepackage{xspace}

\definecolor{darkblue}{cmyk}{1,0.4,0,0.4}  %blue

\newtheorem{proposition}{Proposition}[section]
\newtheorem{theorem}[proposition]{Theorem}
\newtheorem{corollary}[proposition]{Corollary}

\theoremstyle{definition}

\newtheorem{definition}[proposition]{Definition}

\theoremstyle{remark}
\newtheorem{remark}[proposition]{Remark}

\numberwithin{equation}{section}

\crefname{proposition}{Proposition}{Propositions}
\crefname{theorem}{Theorem}{Theorems}
\crefname{corollary}{Corollary}{Corollaries}
\crefname{lemma}{Lemma}{Lemmas}
\crefname{definition}{Definition}{Definitions}
\crefname{remark}{Remark}{Remarks}

\Crefname{proposition}{Proposition}{Propositions}
\Crefname{theorem}{Theorem}{Theorems}
\Crefname{corollary}{Corollary}{Corollaries}
\Crefname{lemma}{Lemma}{Lemmas}
\Crefname{definition}{Definition}{Definitions}
\Crefname{remark}{Remark}{Remarks}

\usepackage{color}

\newcommand{\margincolor}{red}      
\definecolor{darkgreen}{rgb}{0,0.7,0}
\newcommand{\marginauthorcolor}{darkgreen}      

\newcounter{margincounter}
\newcommand{\marginnum}{
\ifnum\value{margincounter}<10
\textcolor{\margincolor}{\begin{picture}(0,0)\put(2.2,2.4){\circle{9}}\end{picture}\footnotesize\arabic{margincounter}}
\else\ifnum\value{margincounter}<100
\textcolor{\margincolor}{\begin{picture}(0,0)\put(4.256,2.5){\circle{11}}\end{picture}\footnotesize\arabic{margincounter}}
\else
\textcolor{\margincolor}{\begin{picture}(0,0)\put(6.8,2.5){\circle{14}}\end{picture}\footnotesize\arabic{margincounter}}
\fi\fi
}

\newcommand{\margin}[2][]
{\!\!\refstepcounter{margincounter}\marginnum\marginpar{\textcolor{\margincolor}{\arabic{margincounter}.}\,\,\textcolor{\marginauthorcolor}{\small#1:}\,\,\tiny #2}}
\newcommand{\marginN}[1]{\margin[NR]{#1}}
\newcommand{\marginE}[1]{\margin[EB]{#1}}
\newcommand{\marginA}[1]{\margin[AT]{#1}}

\newcommand{\response}[2]{ {\color{red}#1:}~#2}
\newcommand{\rn}[1]{\response{NR}{#1}}

\makeatletter
\newcommand{\switchmargin}{
\if@reversemargin
\normalmarginpar
\else
\reversemarginpar
\fi
}
\makeatother

\newcommand{\newword}[1]{\textbf{\emph{#1}}}

\newcommand{\reals}{\mathbb R}

\newcommand{\x}{{\mathbf{x}}}
\newcommand{\ep}{\varepsilon}

\newcommand{\cov}{\operatorname{cov}}
\newcommand{\inv}{\operatorname{inv}}
\newcommand{\Con}{\operatorname{Con}}

\newcommand{\covered}{\lessdot}
\newcommand{\set}[1]{{\left\lbrace #1 \right\rbrace}}

\newcommand{\A}{{\mathcal A}}
\renewcommand{\P}{{\mathcal P}}

\newcommand{\e}{\mathbf{e}}
\newcommand{\join}{\vee}
\newcommand{\meet}{\wedge}
\renewcommand{\Join}{\bigvee}

\newcommand{\orb}{$\times$\xspace}

\title{Noncrossing arc diagrams of type B}
\author{Emily Barnard}
\author{Nathan Reading}
\author{Ashley M. Tharp}
\thanks{Nathan Reading and Ashley Tharp were partially supported by the National Science Foundation under award number DMS-2054489.
}

\begin{document}

\begin{abstract}
Noncrossing arc diagrams are combinatorial models for permutations that encode information about lattice congruences of the weak order and about the associated discrete geometry.
In this paper, we consider two related, analogous models for signed permutations.
One model features centrally symmetric noncrossing arc diagrams, while the other features their quotients modulo the central symmetry.
We demonstrate the utility of the models by applying them to various questions about lattice quotients of the weak order.
% Noncrossing arc diagrams are combinatorial models for permutations that encode information about lattice congruences of the weak order and about the associated discrete geometry.
% In this paper, we construct analogous models for Coxeter groups of type B (signed permutations) and show that these encode the analogous information.
% There are two models, one featuring centrally symmetric noncrossing partitions and another featuring their quotients modulo the symmetry.
% We demonstrate the utility of the models by applying them to various questions about lattice quotients of the weak order.
\end{abstract}
\maketitle

%\vspace{-15pt}

\setcounter{tocdepth}{1}
\tableofcontents  %\marginN{For now, setting toc depth to 3 to see subsubsections, but we might not want to keep this.}

\section{Introduction}\label{intro sec}
Noncrossing arc diagrams~\cite{arcs} are combinatorial objects, in bijection with permutations, that bring to light lattice-theoretic information about the weak order on permutations and related discrete-geometric information.
The results of~\cite{arcs} include:
\begin{itemize}
\item An explicit bijection from permutations to noncrossing arc diagrams and an explicit inverse;
\item a bijection from arcs to join-irreducible elements of the weak order;
\item a bijection from arcs to shards;
\item a characterization of compatibility of join-irreducible elements (in the sense of canonical join complexes);
\item a construction of the canonical join complex of the weak order in a way that proves that this complex is flag; and
\item a characterization of lattice congruences on the weak order and their quotients.
\end{itemize}

In this paper, we consider two models for signed permutations and prove results analogous to the results listed above.
(The theorem on flagness of the canonical join complex was already vastly generalized by the first author in~\cite{Barnard}, so we don't state it separately here.)
In the sequel~\cite{Darcs}, we construct an analogous model for finite Coxeter groups of type D (even-signed permutations) and prove analogous results. 

%In this paper, we construct models analogous to noncrossing arc diagrams for Coxeter groups of type B (signed permutations) and, in a forthcoming paper, type D (even-signed permutations) and prove the analogous results. 
%(The theorem on flagness of the canonical join complex was already vastly generalized by the first author in~\cite{Barnard}, so we don't state it again here for Coxeter groups of types~B and~D.)

%The first challenge in constructing models in types B and D is to choose the correct combinatorial object.
%After that, bijections to elements of the Coxeter groups, bijections to join-irreducible elements and to shards, and characterization of canonical join complexes are all relatively straightforward.
%Significantly more challenging is the combinatorial characterization of lattice congruences and quotients, particularly in type D.
%

The models provide a tool for characterizing lattice congruences and quotients in type B (\cref{B uncontracted,B cong bij} or \cref{orb uncontracted,orb cong bij}).
This tool is a great improvement over the earlier state of the art \cite[Theorem~7.3]{congruence}, which is usable (e.g. in \cite{cambrian,diagram}) but unwieldy.
The improvement is in two directions.
First, the characterization of forcing relations among join-irreducible congruences is a one-step criterion, rather than a description of arrows whose transitive closure is the forcing relation.
Second, the description of the entire forcing relation is much simpler even than the earlier description of single arrows.

%The resulting characterization of congruences in type B is a great improvement over the earlier state of the art \cite[Theorem~7.3]{congruence}, which is usable (e.g. in \cite{cambrian,diagram}) but clunky.
%The improvement is in two directions.
%First, we give a direct one-step characterization of forcing relations among join-irreducible congruences, rather than only characterizing the arrows whose transitive closure is the forcing relation.
%Second, the description of the entire forcing relation is much simpler even than the earlier description of single arrows.

%The characterization of congruences in type D is the first of its kind.
%Although the proof of the type-D characterization is quite complicated, the payoff is that much of the complication can be left behind in the proof, so that the resulting tool is again a one-step characterization and is not prohibitively complicated.

The type-B model comes in two equivalent versions, analogous to the description of signed permutations as sequences $\pi_{-n}\pi_{-n+1}\cdots \pi_{-1}\pi_1\pi_2\cdots \pi_n$ or as sequences $\pi_1\pi_2\cdots \pi_n$.
In the first version, signed permutations of $\set{\pm1,\ldots,\pm n}$ are in bijection with centrally symmetric noncrossing arc diagrams on $2n$ points.
See~\cref{fig:b2symncads}. 
This version has the advantage that the comparison with type-A noncrossing arc diagrams is transparent.  
Centrally symmetric noncrossing arc diagrams appeared in \cite{bicat} and much more extensively in Julian Ritter's thesis \cite{RitterThesis} and in a related paper \cite{shardpoly}.
In particular, the thesis proves the one-step characterization of forcing by a detailed look at the geometry of shards \cite[Proposition~4.2.42]{RitterThesis}.
(See also \cite[Observation~4.2.43]{RitterThesis} and \cite[Theorem~114]{shardpoly}.)
Here, we shortcut some of that process using \cref{Emily's thm}, a general reformulation of forcing among shards that is more easily translated into the language of arcs.
The thesis~\cite{RitterThesis} and paper~\cite{shardpoly} go on to construct quotientopes of types A and B (polytopes realizing lattice quotients of the weak order in the same way the permutohedron realizes the weak order~\cite{quotientopes}).

The second version of the type-B model is an orbifold model, obtained by passing to the quotient of centrally symmetric noncrossing diagrams modulo the symmetry.
See~\cref{fig:b2orbncads}.
In addition to the obvious advantage of spatial compactness, the orbifold model brings to light some more subtle and surprising connections with type-A diagrams:
In the third author's thesis~\cite{TharpThesis} (see also \cite{ReadingTharp}), the meet and join operations in the shard intersection orders of types A and B are characterized in terms of the usual noncrossing arc diagrams and in terms of the orbifold model, and these characterizations are used to show that the shard intersection order of type $A_n$ (permutations in $S_{n+1}$) embeds as a sublattice of the shard intersection order of type~$B_n$.

%The type-D model also has two versions.
%In one version, type-D arcs are equivalence classes of arcs in the type-B orbifold model.
%Most classes are singletons, but some have two elements.
%In the other version, the type-D arcs look less like arcs of type A or B, and come in three types:  ordinary (like type A), partially doubled, and branched.  
%(See~\cref{fig:typeDarcs}.) 
%It would be interesting to make the connection between type-D arcs and the characterization of join-irreducible elements of the type-D weak order produced in a representation-theory context in \cite[Section~6.2]{IRRT} and \cite{Asai}.
%It also seems possible that the existence of a combinatorial model for lattice congruences of the weak order in type D will allow the construction of quotientopes \cite{quotientopes,shardpoly} to be extended to type~D.

%To show the utility of the combinatorial models, we conclude the paper by using the models to compute specific families of congruences in types B and D.
%Examples include Cambrian congruences \cite{cambrian}, bipartite biCambrian congruences \cite{bicat}, and congruences associated to diagram homomorphisms between weak orders \cite{diagram}.  \marginN{Others?}
%Some of these congruences were characterized previously using the previous models, but in those cases the treatment here showcases the improvement in usability in the new models. \marginA{the word "in" shows up three times in this sentence. "improvement in usability of the new models" maybe?}

The paper concludes with examples showing the utility of the orbifold model.
In particular, we give a much simpler proof of a characterization from \cite{diagram} of the generators of congruences of the weak order on~$B_n$ whose quotients are the weak order on $A_n$.
We also greatly simplify the proof of a characterization of Cambrian lattices of type B from \cite{cambrian}, and give the first characterization of the generators of the bipartite and linear biCambrian congruences of type B.

\section{Preliminaries}\label{prelim sec}
In this section, we establish background on lattice theory, finite Coxeter groups, posets of regions, and shards.
We assume the most basic notions, particularly for lattice theory and Coxeter groups.
In Section~\ref{shard sec}, we prove a new characterization of shard arrows, which provides a convenient tool to study forcing among congruences in the case of lattices of regions of a simplicial arrangement.

\subsection{Lattices and congruences}
\label{lat sec}
An element $j$ of a finite lattice $L$ is called \newword{join-irreducible} if and only if it covers exactly one element.  
We write $j_*$ for the unique element covered by $j$.

The \newword{canonical join representation} (CJR) of an element $x$ in a finite lattice~$L$ is the unique antichain $X$ in $L$ such that $x=\Join X$ and such that, if $X'$ is any subset of $L$ with $x=\Join X'$, the order ideal generated by $X'$ contains the order ideal generated by $X$.
In particular, $X$ consists of join-irreducible elements called the \newword{canonical joinands} of $x$.
An element $x$ may or may not have a canonical join representation.

The property that every element of a finite lattice $L$ has a canonical join representation is equivalent to a property called \newword{join-semidistributivity} \cite[Theorem~2.24]{FreeLattices}.
We will not need the usual definition of join-semidistributivity here, but for finite lattices, one can profitably take the existence of canonical join representations as the definition.
If both join-semidistributivity and the dual condition hold, then $L$ is called \newword{semidistributive}.

A natural question, given a finite join-semidistributive lattice $L$, is to characterize which subsets of $L$ are canonical join representations of elements of $L$.
Thus the question is to characterize the collection $\set{X\subseteq L:X\text{ is the CJR of }\Join X}$.
This collection is called the \newword{canonical join complex}, because it is a a simplicial complex (with vertex set the set of join-irreducible elements).

A simplicial complex is \newword{flag} if each minimal non-face is a $2$-element set.
(Equivalently, a set $X$ of vertices forms a face if and only if every $2$-element subset of $X$ is an edge.)
The following theorem is part of \cite[Theorem~2]{Barnard}.

\begin{theorem}\label{flag}
Suppose $L$ is a finite join-semidistributive lattice.
The canonical join complex of $L$ is flag if and only if $L$ is semidistributive.
\end{theorem}

\cref{flag} is extremely important for understanding the combinatorics of finite semidistributive lattices.
Given a finite semidistributive lattice $L$, we say that two join-irreducible elements $j$ and $j'$ of $L$ are \newword{compatible} if and only if $\set{j,j'}$ is a face of the canonical join complex.
(That is, if and only if there is an element whose canonical join representation is $\set{j,j'}$.
Equivalently, if and only if there is an element whose canonical join representation \emph{contains} $\set{j,j'}$.)
Since, in particular,~$L$ is join-semidistributive (so that every element has a canonical join representation), the elements of $L$ are in bijection with the faces of the canonical join complex.
But then since $L$ is semidistributive, \cref{flag} implies that the faces of the canonical join complex are precisely the pairwise compatible sets of join-irreducible elements of $L$.
In later sections, we will give a bijection between join-irreducible elements of the weak order and various kinds of ``arcs'', define notions of compatibility of arcs that correspond to compatibility of join-irreducible elements,
and thus show that elements of the Coxeter group are in bijection with pairwise compatible sets of arcs.

A \newword{congruence} on a lattice $L$ is an equivalence relation $\Theta$ such that if $x_1\equiv x_2$ and $y_1\equiv y_2$ modulo $\Theta$, then $(x_1\meet y_1)\equiv(x_2\meet y_2)$ and $(x_1\join y_1)\equiv(x_2\join y_2)$.
The \newword{quotient} $L/\Theta$ of $L$ modulo $\Theta$ is the lattice whose elements are the $\Theta$-classes $[x]_\Theta$ and whose meet and join are given by $[x]_\Theta\meet[y]_\Theta=[x\meet y]_\Theta$ and $[x]_\Theta\join[y]_\Theta=[x\join y]_\Theta$.

Congruences and quotients have a nice order-theoretic description as well.
An equivalence relation $\Theta$ on a finite lattice $L$ is a lattice congruence if and only if (1) each equivalence class is an interval, (2) the map sending an element to the bottom of its equivalence class is order-preserving, and (3) the map sending an element to the top of its equivalence class is order-preserving.
The quotient $L/\Theta$ is isomorphic to the subposet induced by the set of bottom elements of congruence classes.
Said another way:
An element $x\in L$ is said to be \newword{contracted} by $\Theta$ if $x$ is congruent to some element strictly less than $x$.
The quotient $L/\Theta$ is isomorphic to the subposet of $L$ induced by uncontracted elements.

A join-irreducible element $j$ is contracted by $\Theta$ if and only if $j\equiv j_*$ modulo~$\Theta$.
It is known that a congruence is completely determined by the set of join-irreducible elements it contracts.
(For a precise statement, see \mbox{\cite[Theorem~9-5.12]{regions9}}.)
Thus, to completely characterize congruences on a given lattice, it suffices to characterize which sets of join-irreducible elements can be contracted by congruences.
In this paper, we carry out this characterization in the case of the weak order on Coxeter groups of type B.  % and D.

Canonical join representations behave well when passing to quotients.
The following proposition is \cite[Proposition~10-5.29]{regions9}.

\begin{proposition}\label{CJR quotient}
Suppose $L$ is a finite join-semidistributive lattice and $\Theta$ is a congruence on $L$.
Then an element $x\in L$ is contracted by $\Theta$ if and only if one or more of its canonical joinands is contracted by $\Theta$.
If $x$ is not contracted by $\Theta$, then its canonical join representation in the quotient $L/\Theta$ coincides with its canonical join representation in~$L$.
\end{proposition}
In the proposition, the quotient $L/\Theta$ is realized, as before, as the subposet of~$L$ induced by the elements of $L$ that are not contracted by $\Theta$.
The second assertion of the proposition implies, in particular, that the join-irreducible elements of the quotient are precisely the join-irreducible elements of $L$ that are not contracted by~$\Theta$.

The following corollary is an immediate consequence of \cref{CJR quotient}.
(Recall that the vertices of the canonical join complex of $L$ are the join-irreducible elements of $L$.)

\begin{corollary}\label{CJC quotient}
Suppose $L$ is a finite join-semidistributive lattice and $\Theta$ is a congruence on $L$.
Then the canonical join complex of $L/\Theta$ is the subcomplex of the canonical join complex of $L$ induced by the join-irreducible elements of $L$ not contracted by $\Theta$.
\end{corollary}

Join-irreducible elements cannot be contracted independently.
Instead, there is a partial (pre-)order that mediates which sets of join-irreducibles can be contracted by a congruence.
We say that a join-irreducible element $j_1$ \newword{forces} another join-irreducible element $j_2$ if every congruence that contracts $j_1$ also contracts $j_2$.
In general, the forcing relation is a pre-order (a reflexive, transitive, but not necessarily antisymmetric relation) on the set of join-irreducible elements.

\subsection{Coxeter groups and the weak order}\label{weak sec}
Let $(W,S)$ be a Coxeter system.
We write $\ell(w)$ for the usual length function, the length of $w$ in the alphabet $S$.
A \newword{(left) inversion} of an element $w\in W$ is a reflection $t$ such that $\ell(tw)<\ell(w)$.
Let $\inv(w)$ denote the set of inversions of $w$.
We write ``$\le$'' for the \newword{(right) weak order} on $W$, which is the partial order on $W$ with $u\le w$ if and only if $\inv(u)\subseteq\inv(w)$.
The cover relations in the weak order are precisely the relations $ws\covered w$ such that $w\in W$, $s\in S$, and $\ell(ws)<\ell(w)$.

Given an element $w$ of a Coxeter group $W$, a \newword{cover reflection} of $w$ is a reflection~$t$ such that $tw\covered w$.
Cover reflections of $w$ are in bijection with elements covered by~$w$:  If $ws\covered w$, then $wsw^{-1}$ is a cover reflection.
Let $\cov(w)$ denote the set of cover reflections of $w$.

The weak order on a finite Coxeter group is a semidistributive lattice~\cite{Poly-Barbut}.
The following is \cite[Theorem~10-3.9]{regions10}.  

\begin{theorem}\label{Cox canon}
Suppose $W$ is a finite Coxeter group and $w\in W$.
For each $t\in\cov(w)$, there is a unique minimal element $j_t$ in $\set{v: v\leq w,\,t\in\inv(v)}$.
The canonical join representation of $w$ is $w=\Join\set{j_t: t\in\cov(w)}$.
\end{theorem}

The weak order on a finite Coxeter group has a property called congruence uniformity~\cite{CasPolMor}, which in particular implies that the forcing relation is a partial order.
In later sections, we describe this forcing order on join-irreducible elements of the weak order in types A and B in terms of subarc relationships on arcs.  
In the sequel \cite{Darcs}, we describe the forcing order in type D by an analogous but more complicated notion.

\subsection{Shards, congruences, and forcing}\label{shard sec}
We now briefly recall some discrete-geometric notions that can be used to understand lattice congruences of the weak order, and prove a characterization of forcing that is useful in describing the forcing order on join-irreducible elements.
In fact, we prove the characterization in the more general setting of lattices of regions of simplicial arrangements.  
For more details on the background material, see~\cite{regions9,regions10}.

A \newword{(real, central) hyperplane arrangement} is a finite collection $\A$ of linear hyperplanes in $\reals^d$.
(We us the adjective ``real'' because hyperplane arrangements in vector spaces over other fields are often of interest.
The adjective ``central'' emphasizes that our hyperplanes are linear, as opposed to affine hyperplanes which might not contain the origin.)
The \newword{regions} of $\A$ are the closures of the connected components of $\reals^d\setminus\cup\A$.
Choosing a \newword{base region} $B$, we define the \newword{separating set} $S(R)$ of a region $R$ to be the set of hyperplanes in $\A$ that separate the interior of $R$ from the interior of $B$.
The \newword{poset of regions} $\P(\A,B)$ is the set of regions, partially ordered by setting $Q\le R$ if and only if $S(Q)\subseteq S(R)$.

A hyperplane arrangement is \newword{simplicial} if every region is a simplicial cone (the nonnegative linear span of a linearly independent set of vectors).  
Semidistributivity in the following theorem is \cite[Theorem~3]{hyperplane}.  
The lattice property was proven earlier in \cite[Theorem~3.4]{BEZ}.

\begin{theorem}\label{semid}
If $\A$ is a simplicial hyperplane arrangement, then $\P(\A,B)$ is a semidistributive lattice.
\end{theorem}

The set $\A_W$ of reflecting hyperplanes of a Coxeter group $W$ (in the usual reflection representation) constitute the \newword{Coxeter arrangement} associated to $W$.
Choosing $B$ to be a region bounded by the reflecting hyperplanes for the simple reflections $S$, the map $w\mapsto wB$ is an isomorphism from the weak order on $W$ to the poset of regions $\P(\A_W,B)$.
It is well known that the Coxeter arrangements are simplicial.
(See \cite[Theorem~10-2.1]{regions10} and the citation notes at the end of that chapter.)
Thus \cref{semid} in particular implies the semidistributivity of the weak order, already mentioned in \cref{weak sec}.

A \newword{rank-two subarrangement} of an arrangement $\A$ is a subset $\A'$ of $\A$, with $|\A'|>1$ that can be described as $\A'=\set{H\in\A:H\supset U}$ for some \emph{codimension-$2$} linear subspace $U$.
Any two distinct hyperplanes $H_1,H_2\in\A$ are contained in a unique rank-two subarrangement, namely $\set{H\in\A:H\supset(H_1\cap H_2)}$.

Suppose that we have fixed a base region $B$ as in the definition of $\P(\A,B)$.
Then a rank-two subarrangement has two distinguished hyperplanes called its \newword{basic hyperplanes}.
The subarrangement cuts $\reals^d$ into regions, and $B$ is contained in one of these regions, call it $B'$.
The basic hyperplanes of $\A'$ are the two hyperplanes that bound $B'$.

A hyperplane $H_1\in\A$ \newword{cuts} a hyperplane $H_2\in\A$ if $H_1$ is basic in the rank-two subarrangement $\A'$ containing $H_1$ and $H_2$, and $H_2$ is not basic in $\A'$.
Given $H\in\A$, the \newword{shards} in $H$ are the closures of connected components of $\set{H\setminus H':H'\text{ cuts }H}$.
The set of shards of $\A$ is the set of all shards in all hyperplanes in $\A$.
Thus to make the shards of $\A$, we ``slice'' each hyperplane along all of its intersections with hyperplanes that cut it.
The construction of shards depends on the choice of base region $B$, but when $\A$ is a Coxeter arrangement, all choices of $B$ are symmetric.

Given a shard $\Sigma$, we write $H_\Sigma$ for the hyperplane containing $\Sigma$.
Suppose $R$ is a region of $\A$ and $\Sigma$ is a shard of $\A$.
If $R\cap\Sigma$ is a facet of $R$ and $H_\Sigma\in S(R)$, then $R$ is an \newword{upper region} of $\Sigma$ and $\Sigma$ is a \newword{lower shard} of $R$.
The following proposition is \cite[Propositions~3.2, 3.5]{congruence}.

\begin{proposition}\label{j shard}
If $\A$ is a simplicial arrangement, then each shard $\Sigma$ has a unique minimal upper region $J_\Sigma$.
The region $J_\Sigma$ is join-irreducible in $\P(\A,B)$ and is the unique join-irreducible upper region of $\Sigma$.
The map $\Sigma\mapsto J_\Sigma$ is a bijection from the set of shards of $\A$ to the set of join-irreducible elements of $\P(\A,B)$.
The inverse map takes a join-irreducible region $J$ to its unique lower shard.
\end{proposition}

The following theorem is \cite[Theorem~3.6]{shardint}.

\begin{theorem}\label{shard CJR}
If $\A$ is a simplicial arrangement and $R$ is a region, then the canonical join representation of $R$ in $\P(\A,B)$ is $\set{J_\Sigma:\Sigma\text{ is a lower shard of }R}$.
\end{theorem}

Recall from \cref{lat sec} that two join-irreducible elements $j_1,j_2$ of a semidistributive lattice $L$ are compatible if and only if there exists an element of $L$ whose canonical join representation is $\set{j_1,j_2}$ if and only if there exists an element of $L$ (not necessarily the same element) whose canonical join representation contains $\set{j_1,j_2}$.
We will say that two shards are \newword{compatible} if and only if the intersection of their relative interiors is nonempty.

\begin{proposition}\label{shard compat CJR}
Suppose $\A$ is a simplicial arrangement.
Two shards are compatible if and only if the two corresponding join-irreducible regions in $\P(\A,B)$ are compatible.  
\end{proposition}
\begin{proof}
By \cref{shard CJR}, two join-irreducible regions $J_1$ and $J_2$ in a simplicial poset of regions are compatible if and only if $\Sigma_{J_1}$ and $\Sigma_{J_2}$ are the two lower shards of some region $R$.
In that case, the hyperplanes containing $\Sigma_{J_1}$ and $\Sigma_{J_2}$ are the two basic hyperplanes in the rank-two subarrangement containing them, so they don't cut each other.
Since their intersection has codimension $2$ (because it contains the intersection of two facets of the simplicial region $R$), we see that the relative interiors of $\Sigma_{J_1}$ and $\Sigma_{J_2}$ have nonempty intersection. 
Conversely, if there are two shards $\Sigma_1$ and $\Sigma_2$ whose relative interiors have nonempty intersection, then there is a region~$Q$ having both as lower shards.  
(Such a region $Q$ can be found by starting at a generic point in the intersection of the two relative interiors and moving a small distance in the direction away from the interior of $B$.)
Then \cref{shard CJR} says that the corresponding join-irreducible regions are in the canonical join representation of $Q$, and thus are compatible.
\end{proof}

We have seen that shards and their incidences encode the canonical join representations of regions in $\P(\A,B)$.
We will also see that they encode the forcing (pre-)order on join-irreducible elements in $\P(\A,B)$.
Define the \newword{shard digraph} to be the directed graph whose vertices are the shards, with $\Sigma_1\to\Sigma_2$ if and only if $H_{\Sigma_1}$ cuts $H_{\Sigma_2}$ and $\Sigma_1\cap\Sigma_2$ has codimension $2$. 
The following is \cite[Theorem~9-7.17]{regions9}.

\begin{theorem}\label{shard digraph}
If $\A$ is simplicial and $\Sigma_1$ and $\Sigma_2$ are shards, then $J_{\Sigma_1}$ forces $J_{\Sigma_2}$ if and only if there is a directed path in the shard digraph from $\Sigma_1$ to $\Sigma_2$.
\end{theorem}

Said another way, the map $\Sigma\mapsto J_\Sigma$ is an isomorphism from the reflexive-transitive closure of the shard digraph to the forcing (pre-)order defined in \cref{lat sec}.
We emphasize that, even when the forcing pre-order is a partial order, shard arrows are not necessarily cover relations in the forcing order.
Instead, there may be pairs of shards that are related by a shard arrow and also by a longer path in the shard digraph.

We now state and prove the main theorem of this section, a technical result that rephrases the definition of the shard digraph in a way that is useful (in later sections) for describing forcing in terms of ``subarc'' relationships between arcs.
%more compatible with the construction of noncrossing arc diagrams.

\begin{theorem}\label{Emily's thm}
Suppose $\A$ is a simplicial hyperplane arrangement and suppose $\Sigma_1$ and $\Sigma_2$ are shards.
Then $\Sigma_1\to\Sigma_2$ in the shard digraph if and only if there exists a shard $\Sigma'_1$ satisfying the following conditions:
\begin{enumerate}[\qquad\rm(i)]
\item \label{Sig Sig' compat}
$\Sigma_1$ and $\Sigma'_1$ are compatible,
\item \label{basic nonbasic}
$H_{\Sigma_2}$ is in the rank-two subarrangement containing $H_{\Sigma_1}$ and $H_{\Sigma'_1}$ but is not basic in that subarrangement, and
\item $\Sigma_1\cap\Sigma'_1\subseteq\Sigma_2$.
\end{enumerate}
\end{theorem}
Note that the compatibility of $\Sigma_1$ and $\Sigma_1'$ in \eqref{Sig Sig' compat} implies that there is a region having both $\Sigma_1$ and $\Sigma_1'$ as lower shards.
In particular, $H_{\Sigma_1}$ and $H_{\Sigma'_1}$ are basic in the rank-two subarrangement containing them.

A \newword{shard intersection} is an intersection of a set of shards.  
The \newword{shard intersection order} is the reverse containment order on the set of all shard intersections.
The proof of \cref{Emily's thm} uses the fact that the reverse containment order on the set of all intersections of shards is graded by codimension \cite[Proposition~5.1]{shardint}.
% The proof of \cref{Emily's thm} uses a fact about the shard intersection order, which is \cite[Proposition~5.1]{shardint}:
% The reverse containment order on the set of all intersections of shards is graded by codimension.

\begin{proof}[Proof of \cref{Emily's thm}]
Suppose $\Sigma_1\to\Sigma_2$.
Write $H_1$ for $H_{\Sigma_1}$ and $H_2$ for $H_{\Sigma_2}$.
Then in particular $H_1$ is basic in the rank-two subarrangement $\A'$ containing $H_1$ and $H_2$, but $H_2$ is not basic in $\A'$.
Let $H_1'$ be the other basic hyperplane in $\A'$.
Since $\Sigma_1\cap\Sigma_2$ has codimension $2$ and is contained in the intersection of the hyperplanes of $\A'$, we can find a point $\x\in\Sigma_1\cap\Sigma_2$ that is not contained in any hyperplane in $\A\setminus\A'$.
Thus, since $H_2$ does not cut $H_1$, the point $\x$ is in the relative interior of $\Sigma_1$.
Since $H_2$ also does not cut $H_1'$, $\x$ is also in the relative interior of some shard $\Sigma'_1$ in $H_1'$.
By definition, $\Sigma_1$ and $\Sigma'_1$ are compatible.  
Since $\x$ is in the relative interior of $\Sigma'_1$, there is some ball about $\x$ whose intersection with $\Sigma_1'$ is the same as its intersection with $H'_1$.
Since $H'_1$ contains $\Sigma_1\cap\Sigma_2$, we see that $\Sigma_1\cap\Sigma'_1\cap\Sigma_2$ also has codimension~$2$.
Since $H_1$ and $H'_1$ are not the same hyperplane, $\Sigma_1\cap\Sigma'_1$ has codimension $2$ as well.
Since the shard intersection order is graded by codimension and $\Sigma_1\cap\Sigma'_1\cap\Sigma_2\subseteq\Sigma_1\cap\Sigma'_1$, these two intersections must in fact be the same.
Thus $\Sigma_1\cap\Sigma'_1\subseteq\Sigma_2$.

Conversely, suppose there exists $\Sigma'_1$ such that the three conditions of the theorem hold.
Then certainly $H_{\Sigma_1}$ cuts $H_{\Sigma_2}$, so we must show that $\Sigma_1\cap\Sigma_2$ has codimension~$2$. 
Since $H_{\Sigma_1}\neq H_{\Sigma_2}$, certainly the codimension of $\Sigma_1\cap\Sigma_2$ is at least~$2$.
But the fact that $\Sigma_1$ and $\Sigma'_1$ are compatible means that they intersect in their relative interiors, so $\Sigma_1\cap\Sigma'_1$ has codimension~$2$.
But since $\Sigma_1\cap\Sigma'_1\subseteq\Sigma_2$, we have $\Sigma_1\cap\Sigma'_1=\Sigma_1\cap\Sigma'_1\cap\Sigma_2$.
In particular, the codimension of $\Sigma_1\cap\Sigma_2$ cannot be more than~$2$.
\end{proof}

The third condition in \cref{Emily's thm} can be restated as the existence of a certain order relation in the shard intersection order (in the sense of \cite{shardint}).
In this paper, we will use the third condition directly, without explicitly working with the shard intersection order. 

\section{Noncrossing arc diagrams of type A}\label{A sec}
In this section, we review results of \cite{arcs} on noncrossing arc diagrams for permutations in $S_n$.  %, and describe how \cref{Emily's thm} applies to the proof of one result from~\cite{arcs}.

The Coxeter group of type $A_{n-1}$ can be realized as the group of permutations of $\set{1,\ldots,n}$.  
The \newword{one-line notation} of $\pi\in S_n$ is the sequence $\pi_1\pi_2\cdots \pi_n$ where $\pi_i=\pi(i)$.
One can also realize $S_n$ as a reflection group in $\reals^n$ in the usual way, with each simple reflection $s_i=(i\,\,\,i+1)$ acting as a reflection orthogonal to $\e_{i+1}-\e_i$.
(The $\e_i$ are the standard basis vectors.)
The reflections in $S_n$ are the transpositions $(i\,\,j)$ for $1\le i<j\le n$.

The weak order on $S_n$ has cover relations given by $\pi\covered \sigma$ if $\sigma$ is obtained from $\pi$ by exchanging the entries $\pi_i$ and $\pi_{i+1}$ for some $i\in\set{1,\ldots,n-1}$ such that $\pi_i<\pi_{i+1}$.
The cover reflection of $\sigma$ associated to this cover is $(\pi_i\,\,\,\pi_{i+1})$; multiplying $\pi$ on the left by this reflection swaps the entries $\pi_i$ and $\pi_{i+1}$.
A join-irreducible element of $S_n$ is a permutation whose one-line notation has exactly one descent $\pi_i>\pi_{i+1}$.

We now define noncrossing arc diagrams.
We place $n$ distinct points on a vertical line, identified with the numbers $1,\ldots,n$ from bottom to top.
An \newword{arc} is a curve connecting a point $q\in\set{1,\ldots,n}$ to a strictly lower point $p\in\set{1,\ldots,n}$, moving monotone downwards from $q$ to $p$ without touching any other numbered point, but rather passing to the left of some points and to the right of others.  
A \newword{noncrossing arc diagram} is a collection of arcs that don't intersect, except possibly at their endpoints, such that no two arcs share the same upper endpoint or the same lower endpoint.
The combinatorial data determining an arc consists of which pair of points it connects and which points in between are left or right of the arc.
Two arcs are combinatorially equivalent if they have the same combinatorial data.
We consider arcs and noncrossing arc diagrams up to combinatorial equivalence.
When we need to distinguish these diagrams from the objects defined later (for Coxeter groups of type D), we will refer to them as \newword{type-A arcs} and \newword{type-A noncrossing arc diagrams}. 

%
%We will see that noncrossing arc diagrams on $n$ points are in bijection with permutations of $\set{1,\ldots,n}$, and the purpose of this paper is to describe similar diagrams for signed permutations (type~B) and for even-signed permutations (type~D).
%When the distinction needs to be made, we will refer to the arcs and noncrossing arc diagrams defined above as \newword{type-A arcs} and \newword{type-A noncrossing arc diagrams}.

Noncrossing arc diagrams can also be understood in terms of a compatibility relation on arcs.
We say two arcs are \newword{compatible} if they don't intersect except possibly at one common endpoint, and if they don't share the same upper endpoint or the same lower endpoint.
A noncrossing arc diagram is the same thing as a set of pairwise compatible arcs.
(Certainly the arcs in a noncrossing arc diagram are pairwise compatible, and \cite[Proposition~3.2]{arcs} verifies that any collection of pairwise compatible arcs is combinatorially equivalent to some noncrossing arc diagram.)

We now describe the bijection $\delta$ from $S_n$ to the set of noncrossing arc diagrams on $n$ points.
Given $\pi=\pi_1\cdots\pi_n\in S_n$, write each entry $\pi_i$ at the point $(i,\pi_i)$ in the plane.
For every $i$ such that $\pi_i>\pi_{i+1}$, draw a straight line segment from $\pi_i$ to $\pi_{i+1}$.
These line segments become arcs:  We move the numbers $1,\ldots,n$ horizontally to put them into a single vertical line, allowing the line segments to curve, so that they avoid passing through any numbers and one another.
\begin{figure}
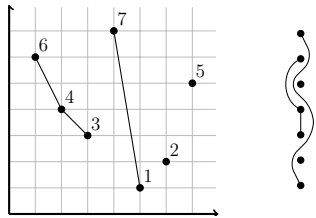

\centering
\includestandalone[height=8em]{figs/a/delta.lines}
\qquad
\includestandalone[height=7em]{figs/a/delta.ncad}
% \includegraphics[width=0.45\linewidth]{classarcs_ex_Adelta.png}
%\caption{The map from permutations to type-A noncrossing arc diagrams for $\pi=6437125$}
\caption{The map $\delta$ applied to $6437125$}
\end{figure}\noindent
We define $\delta(\pi)$ to be the resulting noncrossing arc diagram.

We can alternatively describe $\delta(\pi)$ by listing its arcs:
For every descent $\pi_i>\pi_{i+1}$, there is an arc with endpoints $\pi_{i}$ and $\pi_{i+1}$.  
This arc goes right of every entry $\pi_j$ with $\pi_{i+1}<\pi_j<\pi_i$ and $j<i$ and left of every entry $\pi_j$ with $\pi_{i+1}<\pi_j<\pi_i$ and $j>i$.
The following is \cite[Theorem~3.1]{arcs}.

\begin{theorem}\label{main A}
The map $\delta$ is a bijection from $S_n$ to the set of noncrossing arc diagrams on $n$ points.
\end{theorem}

The proof of \cite[Theorem~3.1]{arcs} includes an explicit description of the inverse map, which we quote here.
(An example is shown in \cref{delta inv fig}.)  
A noncrossing arc diagram has one or more \newword{blocks} (the connected components of the diagram, viewed as an embedded graph).
Each block is a single numbered point or a sequence of arcs sharing endpoints.
Reading each block from top to bottom, we recover the maximal descending runs of the permutation.
The noncrossing arc diagram has at least one \newword{left block}, meaning a block ``with nothing to its left''.
More formally, no other arc or numbered point can be reached from that block by moving horizontally to the left.
The left blocks can be totally ordered from lowest to highest. % by bottom numbered point. 
One can obtain $\pi$ from $\delta(\pi)$ recursively by taking the lowest left block, writing its numbered points in decreasing order and then deleting it from the noncrossing arc diagram.
Recursively, we then remove the lowest left block of what remains and continue writing the one-line notation for $\pi$ from left to right.

\begin{figure}
\def\arraystretch{1.25}
\begin{tabular}{|c|c|c|c|c|c|}
\hline
Step & Start & 1 & 2 & 3 & 4 \\
\hline
Permutation so far &  & \textbf{3} & 3\textbf{86} & 386\textbf{752} & 386752\textbf{41} \\
\hline
Diagram remaining & 
\parbox[c]{3em}{
\includegraphics[width=3em]{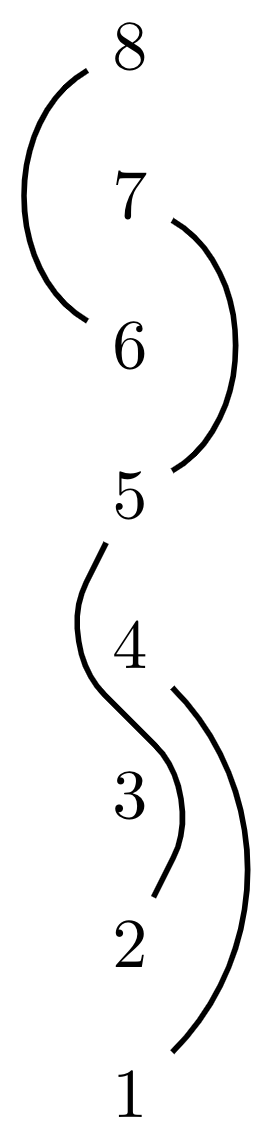}} & 
\parbox[c]{3em}{
\includegraphics[width=3em]{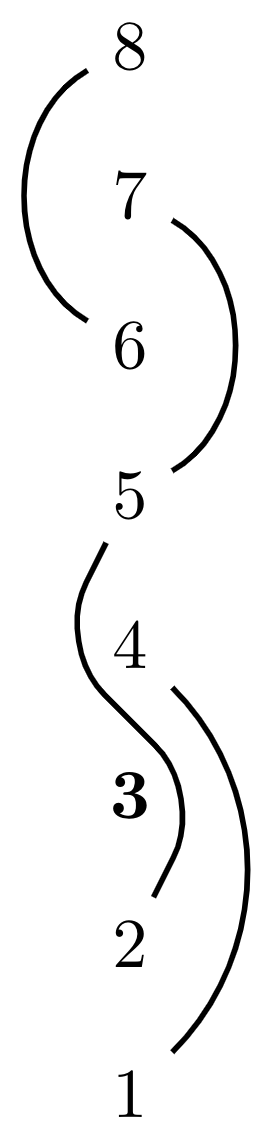}} & 
\parbox[c]{3em}{
\includegraphics[width=3em]{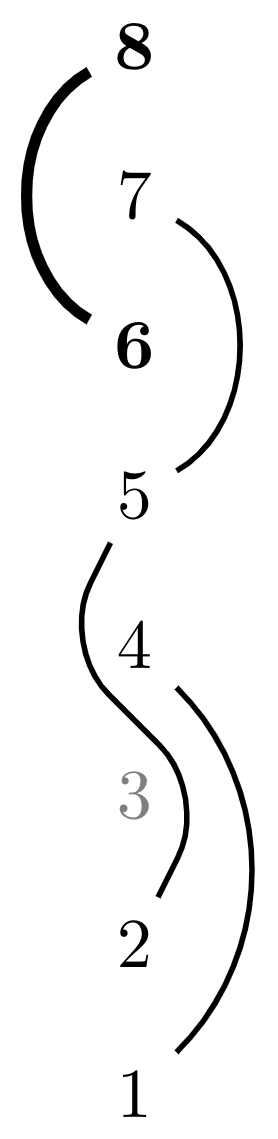}} &
\parbox[c]{3em}{
\includegraphics[width=3em]{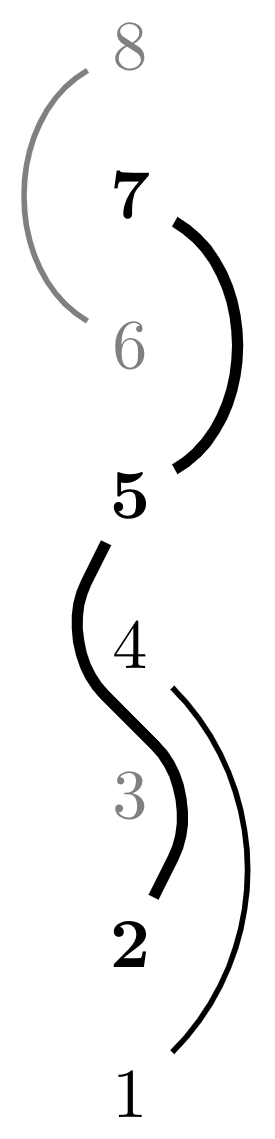}} &
\parbox[c]{3em}{
\includegraphics[width=3em]{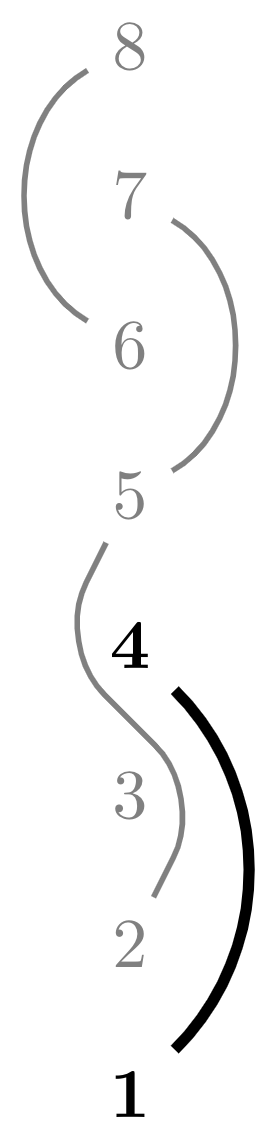}}\\
\hline
\end{tabular}
\caption{The map from type-A noncrossing arc diagrams to permutations}
\label{delta inv fig}
\end{figure} 

Since $\delta$ maps a permutation in $S_n$ with $k$ descents to a noncrossing arc diagram with $k$ arcs, in particular $\delta$ restricts to a bijection from permutations that are join-irreducible in the weak order to single arcs.
Suppose $\alpha$ is an arc connecting point $q$ and $p$ with $q>p$.  
Write $L(\alpha)$ for the set of \newword{left points of $\alpha$} (points in the interval $(p,q)$ that are to the left of $\alpha$) and write $R(\alpha)$ for the set of \newword{right points of $\alpha$} (points in the interval $(p,q)$ that are to the right of~$\alpha$).
Then the join-irreducible element corresponding to $\alpha$ is the permutation that starts with $1\cdots(p-1)$, followed by the elements of $L(\alpha)$ in increasing order, then $q$, then $p$, then the elements of $R(\alpha)$ in increasing order, and finally $(q+1)\cdots n$.

Join-irreducible elements of the weak order are in bijection with shards, and thus shards are in bijection with arcs.
Combining \cite[Proposition 10-5.8]{regions10} with the bijection described above between arcs and join-irreducible permutations, we can write the bijection between arcs and shards as follows.

\begin{proposition}\label{A shard}  
Suppose $\alpha$ is an arc with endpoints $q>p$. %, with left points~$L$, and with right points $R$.
The shard associated to $\alpha$ is 
\[\set{\x\in\reals^n:x_p=x_q\text{ and }x_p\leq x_i\,\forall i\in R(\alpha)\text{ and }x_p\geq x_i\,\forall i\in L(\alpha)}.\]
\end{proposition}

Noncrossing arc diagrams record canonical join representations of permutations, as explained in \cite[Section~3]{arcs} and summarized in the following theorem.
\begin{theorem}\label{A CJR}
Given a permutation $\pi\in S_n$, the canonical join representation of~$\pi$ is the set of join-irreducible elements corresponding to the set of arcs in $\delta(\pi)$.
\end{theorem}

Recall that two join-irreducible elements in a semidistributive lattice are called compatible if and only if they can appear together in a canonical join representation.
As a consequence of \cref{A CJR}, two join-irreducible elements of $S_n$ are compatible if and only if the corresponding arcs are compatible.

\begin{definition}\label{A subarc}
An arc $\alpha'$ with endpoints $q'>p'$ is a \newword{subarc} of an arc $\alpha$ with endpoints $q>p$ if and only if $q\ge q'>p'\ge p$ and $R(\alpha')=R(\alpha)\cap(p',q')$.
If $\alpha'$ is a subarc of $\alpha$, then we say $\alpha$ is a \newword{superarc} of $\alpha'$.
Thus, given $\alpha'$, a superarc $\alpha$ is obtained by pushing $\alpha'$ right or left of
the top and/or bottom endpoint, independently, and then extending upward and/or downward to make a longer arc.
\end{definition}

\begin{remark}\label{A explicit}
It will be convenient in later generalizations to construct subarcs explicitly.
Define a function $h$ from the plane to $\reals$ that returns the vertical height of a point, calibrated so that each numbered point $i$ has height $i$.
Given an arc $\alpha$, parametrized as a function from the interval $[0,1]$ into the plane, we obtain a subarc $\alpha'$ of $\alpha$ as follows:
Choose $t_1$ and $t_2$ with $0\le t_1<t_2\le1$ such that $h(\alpha(t_1))\in\set{1,\ldots,n}$ and $h(\alpha(t_2))\in\set{1,\ldots,n}$.
Also choose $\ep_1>0$ and $\ep_2>0$ such that $|h(\alpha(t_1+\ep_1))-h(\alpha(t_1))|<1$, $|h(\alpha(t_2-\ep_2))-h(\alpha(t_2))|<1$, and $t_1+\ep_1<t_2-\ep_2$.
Define $\alpha'$ to be the arc obtained by concatenating three curves:
First, the straight line segment from the point numbered $h(\alpha(t_1))$ to the point $\alpha(t_1+\ep_1)$;
second, the restriction of $\alpha$ to the interval $[t_1+\ep_1,t_2-\ep_2]$;
and third, the straight line segment from $\alpha(t_2-\ep_2)$ to the point numbered $h(\alpha(t_2))$.
\end{remark}

The following theorem is \cite[Theorem~4.4]{arcs}.
A new proof of the theorem using \cref{Emily's thm} is found in \cite[Section~2.3]{TharpThesis}. 
%, but we prove it again here 
%because we will reuse the argument in type D and

\begin{theorem}\label{arc forcing A}
Let $j_1$ and $j_2$ be join-irreducible permutations, corresponding to arcs $\alpha_1$ and $\alpha_2$ respectively.
Then $j_1$ forces $j_2$ if and only if $\alpha_1$ is a subarc of $\alpha_2$.
\end{theorem}

%Before proving the theorem, we state two important corollaries.
We state two important corollaries of \cref{arc forcing A}, both of which are stated as part of \cite[Corollary~4.5]{arcs}.
The first is simply a rephrasing of \cref{arc forcing A}.
The second is obtained by combining \cref{A CJR,main A} with \cref{CJC quotient},

\begin{corollary}\label{A uncontracted}
A set $U$ of arcs corresponds to the set of \emph{un}contracted join-irreducible permutations of some congruence $\Theta$ on $S_n$ if and only if $U$ is closed under passing to subarcs.
\end{corollary}

\begin{corollary}\label{A cong bij}
If $\Theta$ is a congruence on $S_n$ and $U$ is the set of arcs corresponding to join-irreducible permutations \emph{not} contracted by $\Theta$, then $\delta$ restricts to a bijection from the quotient $S_n/\Theta$ (the set of  permutations not contracted by $\Theta$) to the set of noncrossing arc diagrams consisting only of arcs in $U$.
\end{corollary}

We conclude this section with a result that will be helpful in \cref{B sym sec,B orb sec}.
A finite Coxeter group has an element $w_0$ that is longer than every other element (in the usual sense of length in terms of number of letters in a reduced word or number of inversions).
The element $w_0$ is an involution.

In $S_n$, this element $w_0$ is $n(n-1)\cdots321$.
Given $\pi=\pi_1\pi_2\cdots \pi_n\in S_n$, $\pi w_0$ is the permutation with one-line notation $\pi_n\pi_{n-1}\cdots \pi_1$ and $w_0\pi$ is the permutation with one-line notation $(n+1-\pi_1)(n+1-\pi_2)\cdots(n+1-\pi_n)$.
The one-line notation of $w_0\pi w_0$ is $(n+1-\pi_n)(n+1-\pi_{n-1})\cdots(n+1-\pi_1)$.

In the following proposition, we assume that the points are placed so that a rotation by a half turn sends the point labeled $i$ to the point labeled $n+1-i$.

\begin{proposition}\label{w0 arc}
If $\pi\in S_n$, then the arc diagrams $\delta(\pi)$ and $\delta(w_0\pi w_0)$ are related by a half turn that sends the point labeled $i$ to the point labeled $n+1-i$.
\end{proposition}
\begin{proof}
There is a descent $\pi_i>\pi_{i+1}$ in $\pi$ if and only if there is a descent $(n+1-\pi_{i+1})>(n+1-\pi_i)$ in $w_0\pi w_0$, and an entry $a$ with $\pi_{i+1}<a<\pi_{i+1}$ occurs left of $\pi_i$ in $\pi$ if and only if $(n+1-a)$ is right of $(n+1-\pi_i)$ in $w_0\pi w_0$. 
Thus $\delta(w_0\pi w_0)$ is obtained from $\delta(\pi)$ by a \emph{reflection} taking each point labeled $i$ to the point labeled $n+1-i$, followed by a reflection in the vertical line containing the points.
The composition of these two reflections is a half turn that sends each point labeled $i$ to the point labeled $n+1-i$. 
\end{proof}

\section{Centrally symmetric noncrossing arc diagrams}\label{B sym sec}
In this section and the next, we establish two notions of noncrossing arc diagrams for Coxeter groups of type B.
While these notions are in some sense equivalent, they look quite different.
We begin, in this section, with a centrally symmetric model, after briefly recalling some background about Coxeter groups of type $B_n$ and establishing notation.

The Coxeter group of type $B_n$ can be realized as the group of signed permutations of $\set{\pm1,\ldots,\pm n}$.  
These are the permutations $\pi$ of $\set{\pm1,\ldots,\pm n}$ with the property that $\pi(-i)=-\pi(i)$ for $i=1,\ldots,n$.
The \newword{long one-line notation} of $\pi\in B_n$ is the sequence $\pi_{-n}\pi_{-n+1}\cdots \pi_{-1}\pi_1\pi_2\cdots \pi_n$.
But $\pi\in B_n$ is completely determined by its \newword{short one-line notation} (or simply \newword{one-line notation}) $\pi_1\pi_2\cdots \pi_n$.

We realize $B_n$ as usual as a reflection group in $\reals^n$, with $s_0$ acting as a reflection orthogonal to the standard basis vector $\e_1$ and with each $s_i$ acting as a reflection orthogonal to $\e_{i+1}-\e_i$ for $i=1,\ldots,n-1$.
For convenience, we write $\e_{-i}$ to mean $-\e_i$ for each $i=1,\ldots,n$.
Given a vector $\x=(x_1,\ldots,x_n)\in\reals^n$, we write $x_{-i}$ to mean $-x_i$. 

The reflections in $B_n$ are the permutations with cycle notation $(i\,\,-i)$ for $i\in\set{\pm1,\ldots,\pm n}$ or $(i\,\,j)(-i\,\,-j)$ for $i,j\in\set{\pm1,\ldots,\pm n}$, and $|j|\neq |i|$.

The weak order on $B_n$ has cover relations given by $\pi\covered \sigma$ if $\sigma$ is obtained in one of two ways from $\pi$. 
One way is that $\pi_{-1}<\pi_1$ and $\sigma$ is obtained by exchanging $\pi_{-1}$ and $\pi_1$.
The cover reflection of $\sigma$ associated to this cover is the involution $(\pi_1\,\,\,\pi_{-1})$.
The other way is that $\pi_i<\pi_{i+1}$ for some $i\in\set{1,\ldots,n-1}$ (and equivalently $\pi_{-i-1}<\pi_{-i}$) and $\sigma$ is obtained from $\pi$ by exchanging $\pi_{i}$ and $\pi_{i+1}$ and exchanging $\pi_{-i-1}$ and $\pi_{-i}$.
The cover reflection of $\sigma$ associated to this cover is $(\pi_i\,\,\,\pi_{i+1})(\pi_{-i}\,\,\,\pi_{-i-1})$.

A join-irreducible element of $B_n$ is a signed permutation whose long one-line notation either has exactly one descent $\pi_{-1}>\pi_1$ or has exactly two descents that are symmetric to each other: $\pi_i>\pi_{i+1}$ and $\pi_{-i-1}>\pi_{-i}$ for some $i\in\set{1,\ldots,n-1}$.
Equivalently, either $\pi_1<0$ and the (short) one-line notation is increasing or $\pi_1>0$ and there is exactly one $i\in\set{1,\ldots,n-1}$ such that $\pi_i>\pi_{i+1}$.

%\subsection{A centrally symmetric model}\label{B sym sec}
In the centrally symmetric model, we begin with $2n$ distinct points on a vertical line, and the points identified with the numbers $-n,\ldots, -2,-1,1,2,\ldots, n$, in order, with $-n$ at the bottom.  
The points are placed so that the antipodal map $x\mapsto-x$ (a half turn about the origin) takes $i$ to $-i$ for all $i\in\set{\pm1,\ldots,\pm n}$. 
A \newword{centrally symmetric noncrossing arc diagram} is a collection of arcs on these points, satisfying the same requirements as in type A, with the additional requirement that the entire diagram is symmetric with respect to the half-turn symmetry.  
All centrally symmetric noncrossing arc diagrams for $B_2$ are shown in \cref{fig:b2symncads}.
\begin{figure}
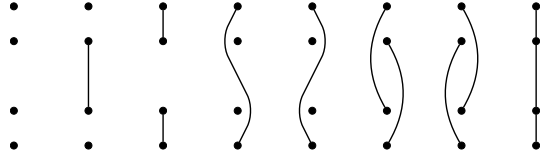

    \centering\includestandalone[height=6em]{figs/b/b2.sym.allncads}
    \caption{Centrally symmetric noncrossing arc diagrams for $B_2$}
    \label{fig:b2symncads}
\end{figure}

A centrally symmetric noncrossing arc diagram consists of \newword{symmetric arcs} (arcs that are fixed by the central symmetry) and \newword{symmetric pairs of arcs} (pairs of arcs that are compatible in the type-A sense of \cref{A sec} and are mapped to each other by the central symmetry).
Symmetric pairs of arcs come in two types:
A \newword{non-overlapping symmetric pair}
is a symmetric pair $\set{\alpha,-\alpha}$ such that $\alpha$ connects positive points and $-\alpha$ connects negative points, so that neither arc is left or right of the other.
An \newword{overlapping symmetric pair} is a symmetric pair $\set{\alpha,-\alpha}$ in which each arc connects a positive point to a negative point.
In an overlapping pair, since the two arcs are compatible, one is to the right of the other.

For the purposes of this section, we can harmlessly recast the results of \cref{A sec} in terms of a group $S_{2n}$ of permutations of $\set{\pm1,\ldots,\pm n}$ and in terms of noncrossing arc diagrams on  points labeled $-n,\ldots,-1,1,\ldots,n$.
As an immediate consequence of \cref{main A,w0 arc}, we have the following theorem.
\begin{theorem}\label{main B}
The map $\delta$ restricts to a bijection from $B_n$ to the set of centrally symmetric noncrossing arc diagrams on $2n$ points.
\end{theorem}

\begin{figure}
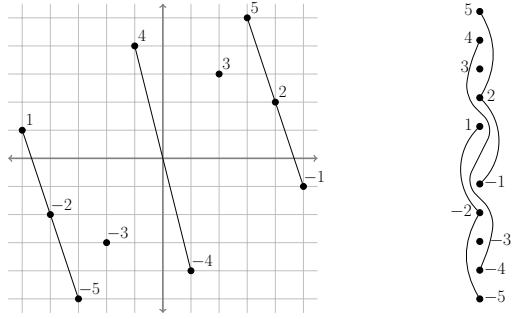

\includestandalone[height=12em]{figs/b/sym.draw}
\qquad\qquad
\includestandalone[height=12em]{figs/b/sym.smoosh}
\caption{The map $\delta$ applied to $\pi=(-4)352(-1)$}
\label{fig:b.map.sym}
\end{figure}

An example of $\delta$ applied to a signed permutation in $B_5$ is shown in \cref{fig:b.map.sym}.

A join-irreducible signed permutation is mapped by $\delta$ to an arc diagram with only one arc (a symmetric arc), or only two arcs (a symmetric pair), and this is a bijection from join-irreducible signed permutations to symmetrics arcs/pairs.  
We describe this bijection explicitly below as a map from symmetric arcs/pairs to join-irreducible signed permutations in (short) one-line notation.
In this description, when we write a set as part of the one-line notation, we mean the elements of that set in increasing order.
Also in this description, $i\cdots j$ will always stand for the sequence of elements increasing by 1 from $i$ to $j$, or the empty sequence if $j=i-1$.

Recall from \cref{A sec} that the set $L(\alpha)$ of \newword{left points} of an arc $\alpha$ is the set of numbers that are left of $\alpha$, and the set $R(\alpha)$ of \newword{right points} of $\alpha$ is the set of numbers that are right of $\alpha$.
(Applying this definition to arc diagrams on points $\set{\pm1,\ldots,\pm n}$, of course $0$ is never a right point or a left point.)
We write $-L(\alpha)$ for $\set{-i:i\in L(\alpha)}$ and similarly $-R(\alpha)$.

A join-irreducible signed permutation $\pi$ with one descent $\pi_{-1}>\pi_1$ corresponds to a symmetric arc $\alpha$.
If $\alpha$ has endpoints $-p<p$, then $\pi$ is 
\[(-p)\,R(\alpha)\,(p+1)\cdots n.\]
A join-irreducible signed permutation $\pi$ that has two descents $\pi_i>\pi_{i+1}$ and ${\pi_{-i-1}>\pi_{-i}}$ corresponds to a symmetric pair $\set{\alpha,-\alpha}$ of arcs.
Suppose $\alpha$ has endpoints $p<q$.
If $\set{\alpha,-\alpha}$ is nonoverlapping, then we assume that $\alpha$ is above $-\alpha$, so that $0<p<q$.
Then $\pi$ is 
\[1\cdots(p-1)\,L(\alpha)\,q\,p\,R(\alpha)\,(q+1)\cdots n.\]
If $\set{\alpha,-\alpha}$ is overlapping, then we assume that $\alpha$ is right of $-\alpha$.
If $p<0<-p<q$, then $\pi$ is
\[[(0,-p)\cap L(\alpha)\cap-L(\alpha)]\,[(-p,q)\cap L(\alpha)]\,q\,p\,R(\alpha)\,(q+1)\cdots n\]
If $p<0<q<-p$, then $\pi$ is
\[[(0,q)\cap L(\alpha)\cap-L(\alpha)]\,q\,p\,R(\alpha)\,[(q,-p)\cap-L(\alpha)]\,(-p+1)\cdots n.\]

Given a signed permutation $\pi\in B_n$, each symmetric arc or symmetric pair of arcs in $\delta(\pi)$ is associated to a cover reflection $t$ of $\pi$.
Each symmetric arc or symmetric pair also specifies a join-irreducible element $j$ as described above.
By inspection of the three cases described above, we see that $j\le \pi$ and that $j$ is minimal with respect to the property that $t\in \inv(j)$.
Thus \cref{Cox canon} implies the following theorem.

\begin{theorem}\label{B CJR}
Given $\pi\in B_n$, the canonical join representation of $\pi$ is the set of join-irreducible elements corresponding to the symmetric arcs and symmetric pairs of arcs in $\delta(\pi)$.
\end{theorem}

Recall from \cref{A sec} that two arcs are called compatible if and only if they don't intersect, except possibly at endpoints, and don't share the same top endpoint or the same bottom endpoint.
We now define one symmetric arc/pair $A$ to be \newword{compatible} with another symmetric arc/pair $A'$ if and only if every arc in $A$ is compatible with every arc in $A'$.
As a consequence of \cref{B CJR}, two symmetric arcs/pairs are compatible if and only if the corresponding join-irreducible elements are compatible.

The following two propositions are obtained by combining the correspondence between arcs and join-irreducible signed permutations with simple observations in \cite[Sections~3,5]{congruence}.
In those sections, join-irreducible elements are shown to be in bijection with certain ``signed subsets'', and the signed subsets are used to write inequalities for shards.

\begin{proposition}\label{B shard sym arc}
Suppose $\alpha$ is a symmetric arc having endpoints $-p$ and~$p$.
The shard associated to $\alpha$ is 
\[\set{\x\in\reals^n:x_p=0\text{ and }0\leq x_i\,\forall i\in R(\alpha)}.\]
\end{proposition}

\begin{proposition}\label{B shard sym pair}  
Suppose $\set{\alpha,-\alpha}$ is a symmetric pair of arcs such that $\alpha$ has endpoints $p$ and $q$.
%Let $L$ be the set of left points of $\alpha$ and let $R$ be the set of right points of $\alpha$. 
The shard associated to $\set{\alpha,-\alpha}$ is 
\[\set{\x\in\reals^n:x_p=x_q\text{ and }x_p\leq x_i\,\forall i\in R(\alpha)\text{ and }x_p\geq x_i\,\forall i\in L(\alpha)}.\]
\end{proposition}

\begin{remark}\label{my bad}
There is a global change in the direction of inequalities between \cref{B shard sym arc,B shard sym pair} and \cite{congruence}, because in \cite{congruence}, hyperplanes were identified with their normal vectors pointing \emph{away} from a ``base region'' $B$.
A standard convention for root systems has the positive roots pointing \emph{towards} a ``fundamental chamber''.
This convention was used in \cite{regions10}, from which we quoted \cref{A shard}, so to keep conventions consistent in this paper, we apply the antipodal map relative to \cite{congruence}.  
\end{remark}

We now use \cref{Emily's thm} to characterize forcing of join-irreducible signed permutations in terms of a centrally symmetric version of subarcs, which we now define.
Since join-irreducible signed permutations correspond either to centrally symmetric arcs or to centrally symmetric pairs of arcs, we will use the terminology of a ``subarc pair'' of a symmetric arc or symmetric pair, or a ``subarc'' of a symmetric arc or symmetric pair.
\cref{fig:subarcs sym} shows subarcs and subarc pairs of a symmetric arc and an overlapping symmetric pair. 

\begin{figure}
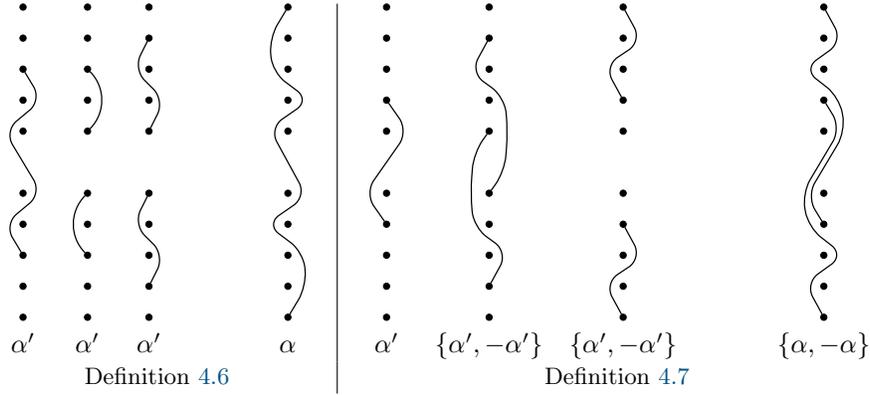

\begin{tabular}{c|c}
\begin{tabular}{cccp{18pt}c}
\includestandalone[height=12em]{figs/b/arrow.S.sym.sym.sub}
&\includestandalone[height=12em]{figs/b/subarc.S.sym.non}
&\includestandalone[height=12em]{figs/b/subarc.S.sym.non1}
&&\includestandalone[height=12em]{figs/b/arrow.S.sym.big}
\\
$\alpha'$&$\alpha'$&$\alpha'$&&$\alpha$
\end{tabular}
&
\begin{tabular}{cccp{18pt}c}
\includestandalone[height=12em]{figs/b/arrow.S.over.sym.sym}
&\includestandalone[height=12em]{figs/b/subarc.S.over.mid}
&\includestandalone[height=12em]{figs/b/arrow.S.over.sym.non}
&&\includestandalone[height=12em]{figs/b/arrow.S.over.sym.big}
\\
$\alpha'$&$\{\alpha',-\alpha'\}$&$\{\alpha',-\alpha'\}$&&$\{\alpha,-\alpha\}$
\end{tabular}
\\[50pt]
\small\cref{subs of sym}
&
\small\cref{subs of pair}
\end{tabular}
\caption{Subarcs and subarc pairs}
\label{fig:subarcs sym}
\end{figure}

\begin{definition}[Subarcs/subarc pairs of a symmetric arc]\label{subs of sym}
Suppose $\alpha$ is a symmetric arc with endpoints ${-p<p}$.
A \newword{subarc} of $\alpha$ is a symmetric arc $\alpha'$ with endpoints $-p'<p'$ with $p'\le p$ and $R(\alpha')=R(\alpha)\cap(-p',p')$.
A \newword{subarc pair} of $\alpha$ is a non-overlapping pair $\set{\alpha',-\alpha'}$ of arcs such that $\alpha'$ has endpoints $p'$ and $q'$ with $0<p'<q'\le p$ and $R(\alpha')=R(\alpha)\cap(p',q')$.
\end{definition}

\begin{definition}[Subarcs/subarc pairs of a symmetric pair of arcs]\label{subs of pair}
Suppose $\set{\alpha,-\alpha}$ is a symmetric pair of arcs and $\alpha$ has endpoints $p<q$.
A \newword{subarc} of $\set{\alpha,-\alpha}$ is a symmetric arc $\alpha'$ with endpoints $-p'$ and $p'$ having $p\le-p'<p'\le q$ and $R(\alpha')=R(\alpha)\cap(-p',p')=R(-\alpha)\cap(-p',p')$.
%$R(\alpha')=R(\alpha)\cap(-p',p')=R(-\alpha')\cap(-p',p')$.
This can only happen if $\alpha$ and $-\alpha$ are overlapping there are no numbered points between them on the interval $(-p',p')$.
Now, if $\set{\alpha,-\alpha}$ is overlapping, suppose further that $\alpha$ is right of $-\alpha$.
A \newword{subarc pair} of $\set{\alpha,-\alpha}$ is a symmetric pair $\set{\alpha',-\alpha'}$ with endpoints $p'$ and $q'$ with $p\le p'<q'\le q$ and $R(\alpha')=R(\alpha)\cap(p',q')$, satisfying an additional requirement: If $\set{\alpha',-\alpha'}$ is overlapping then $\alpha'$ is also right of $-\alpha'$.
The left two pictures of \cref{fig:subarc fail} show a failure of this additional requirement. 
\end{definition}

\begin{figure}
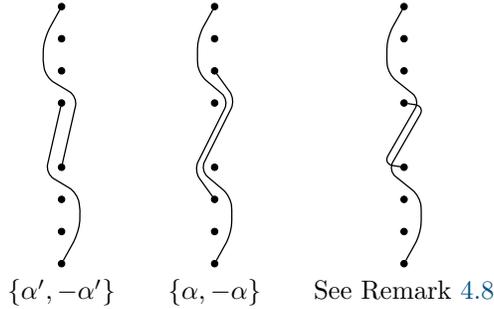

    \centering
    \begin{tabular}{ccccc}
\includestandalone[height=10em]{figs/b/subarc.S.counterex.small} 
&&\includestandalone[height=10em]{figs/b/subarc.S.counterex.big} 
&&\includestandalone[height=10em]{figs/b/subarc.S.counterex.concat} \\
 $\{\alpha',-\alpha'\}$ 
&&$\{\alpha,-\alpha\}$ 
&& See \cref{sym explicit}
    \end{tabular}
    \caption{A pair not satisfying the additional requirement of \cref{subs of pair}}
    \label{fig:subarc fail}
\end{figure}

\begin{remark}\label{sym explicit}
To understand this additional requirement, it is useful to construct subarcs of $\set{\alpha,-\alpha}$ explicitly.
Similarly to the type-A construction, we take $h$ to be a function from the plane to $\reals$ that returns the vertical height of a point, calibrated so that each numbered point $i\in\set{\pm1,\ldots,\pm n}$ has height $i$.
Choose $t_1$ and $t_2$ with ${0\le t_1<t_2\le1}$ such that $h(\alpha(t_1)),h(\alpha(t_2))\in\set{\pm1,\ldots,\pm n}$ and choose $\ep_1>0$ and $\ep_2>0$ with $|h(\alpha(t_1+\ep_1))-h(\alpha(t_1))|<1$, $|h(\alpha(t_2-\ep_2))-h(\alpha(t_2))|<1$, and $t_1+\ep_1<t_2-\ep_2$.
Then $\alpha'$ is obtained by concatenating three curves: the segment from the point numbered $h(\alpha(t_1))$ to the point $\alpha(t_1+\ep_1)$, the restriction of $\alpha$ to $[t_1+\ep_1,t_2-\ep_2]$ and the segment from $\alpha(t_2-\ep_2)$ to the point numbered $h(\alpha(t_2))$.
The additional requirement on $\set{\alpha',-\alpha'}$ is that (up to choosing another symmetric arc pair combinatorially equivalent to $\set{\alpha,-\alpha}$) the specific curves $\alpha'$ and $-\alpha'$ constructed here do not intersect each other.
The right picture of \cref{fig:subarc fail} shows an example of $\alpha'$ and $-\alpha'$ crossing. 
\end{remark}

We will prove the following theorem.

\begin{theorem}\label{arc forcing B}
Let $j_1$ and $j_2$ be join-irreducible signed permutations.
Then $j_1$ forces $j_2$ if and only if the arc or pair of arcs corresponding to $j_1$ is a subarc or subarc pair of the arc or pair of arcs corresponding to $j_2$.
\end{theorem}

\cref{arc forcing B} lets us prove the type-B analogues of \cref{A uncontracted,A cong bij}.
The following result is a rephrasing of \cref{arc forcing B}.

\begin{corollary}\label{B uncontracted}
A set $U$ of symmetric arcs/symmetric pairs corresponds to the set of \emph{un}contracted join-irreducible signed permutations of some congruence $\Theta$ on $B_n$ if and only if $U$ is closed under passing to subarcs/subarc pairs.
\end{corollary}

Combining \cref{B uncontracted} with \cref{B CJR,main B} and \cref{CJC quotient}, we obtain the following result.

\begin{corollary}\label{B cong bij}
If $\Theta$ is a congruence on $B_n$ and $U$ is the set of symmetric arcs/pairs corresponding to join-irreducible signed permutations \emph{not} contracted by $\Theta$, then~$\delta$ restricts to a bijection from the quotient $B_n/\Theta$ (the set of  signed permutations not contracted by $\Theta$) to the set of centrally symmetric noncrossing arc diagrams consisting only of arcs in $U$.
\end{corollary}

It is convenient to understand the opposite of the subarc relation.
We say that a symmetric arc/pair $A_2$ is a \newword{superarc} or \newword{superarc pair} of a symmetric arc/pair $A_1$ if and only if $A_1$ is a subarc/subarc pair of $A_2$.
We now describe how to construct a superarc/superarc pair of a given symmetric arc/pair.
The construction is a direct rephrasing of \cref{subs of sym,subs of pair}.

\medskip
\noindent
\textbf{Superarcs/superarc pairs of a symmetric arc $\alpha$.}
We construct a superarc by pushing $\alpha$ right or left of its top endpoint, symmetrically pushing it left or right of its bottom endpoint, and extending $\alpha$ symmetrically upwards and downwards.
We construct a superarc pair by first replacing $\alpha$ with two arcs that are combinatorially equivalent to $\alpha$, antipodal images of each other, and disjoint except at their endpoints.
We then push one or both copies right or left of the top endpoint, independently but without making the curves cross each other, make the symmetric change at the bottom endpoint, and extend the curves upward and downward symmetrically.

\medskip
\noindent
\textbf{Superarcs/superarc pairs of a non-overlapping symmetric pair of arcs $\set{\alpha,-\alpha}$.}
To construct a superarc or superarc pair, we first push $\alpha$ and $-\alpha$ left or right of their inner endpoints and/or outer endpoints in a way that preserves the symmetry.
We construct a superarc by extending the curves inward and connecting them, and also possibly extending both outwards. 
We construct a superarc pair by extending the curves symmetrically to create a symmetric pair of compatible arcs, either overlapping or not.

\medskip
\noindent
\textbf{Superarc pairs of an overlapping symmetric pair of arcs $\set{\alpha,-\alpha}$.}
(There are no superarcs of $\set{\alpha,-\alpha}$ and no non-overlapping superarc pairs, only overlapping superarc pairs.)
We construct a superarc pair by pushing $\alpha$ right or left of one or both endpoints independently, making the symmetric change to $-\alpha$, and then extending the curves symmetrically to create a symmetric pair of compatible overlapping arcs.
The additional requirement in \cref{subs of pair} is implied by this description of superarcs:
If $\set{\alpha_1,-\alpha_1}$ fails to be a subarc pair of $\set{\alpha_2,-\alpha_2}$ because it fails the additional requirement, then the attempt to extend $\alpha_1$ and $-\alpha_1$ to obtain $\alpha_2$ and $-\alpha_2$ will result in a pair of arcs that cross each other.
(See \cref{fig:subarc fail}.)

A first step towards the proof of \cref{arc forcing B} is the following observation.

\begin{proposition}\label{B sub trans}
The subarc/subarc pair relation on symmetric arcs/pairs is transitive.
\end{proposition}
The proposition is verified by checking the various cases (e.g.\ a nonoverlapping subarc pair of a symmetric subarc of an overlapping arc pair).
There are many cases to check, but each individual case is easy.  

As a next step towards proving \cref{arc forcing B}, we use \cref{Emily's thm} to characterize arrows in the shard digraph in terms of arcs and arc pairs.
Given symmetric arcs/pairs $A_1$ and $A_2$, we write $A_1\to A_2$ if and only if the corresponding shards $\Sigma_1$ and $\Sigma_2$ have $\Sigma_1\to\Sigma_2$ in the shard digraph.

\begin{proposition}\label{Emily's B arcs}
Suppose $A_1$ is a symmetric arc or pair and $A_2$ is another symmetric arc or pair, then $A_1\to A_2$ if and only if $A_1$ is a subarc/subarc pair of $A_2$ and one of the following conditions holds. 
\begin{enumerate}[\qquad\rm(i)]
\item \label{pair pair}
$A_2$ is a symmetric pair $\set{\alpha_2,-\alpha_2}$ such that $\alpha_2$ has endpoints $p$ and $q$ with $p<q$ and $A_1$ is a symmetric pair $\set{\alpha_1,-\alpha_1}$ such that $\alpha_1$ has endpoints $p'$ and $q$ with $p'\neq-p$ or $\alpha_1$ has endpoints $p$ and $q'$ with $q'\neq-q$. 
\item \label{sym sym}
$A_1$ and $A_2$ are both symmetric arcs with no endpoints in common.
\item \label{sym non}
$A_2$ is a symmetric arc $\alpha_2$ with endpoints $p$ and $-p$ and $A_1$ is a non-overlapping symmetric pair $\set{\alpha_1,-\alpha_1}$ such that $\alpha_1$ has an endpoint $p$.
\item \label{over sym}
$A_2$ is a symmetric pair $\set{\alpha_2,-\alpha_2}$ such that $\alpha_2$ has endpoints $p$ and $q$ with ${-q<p<0<-p<q}$ and such that each point in $\set{\pm1,\ldots,\pm(p-1)}$ is left of $\alpha_2$ if and only if it is left of $-\alpha_2$, and $A_1$ is a symmetric arc with endpoints $p$ and $-p$.
\item \label{over non} 
$A_2$ is a symmetric pair $\set{\alpha_2,-\alpha_2}$ such that $\alpha_2$ has endpoints $p$ and $q$ with ${-q<p<0<-p<q}$ and such that each point in $\set{\pm1,\ldots,\pm(p-1)}$ is left of $\alpha_2$ if and only if it is left of $-\alpha_2$, and $A_1$ is a non-overlapping symmetric pair $\set{\alpha_1,-\alpha_1}$ such that $\alpha_1$ has endpoints $-p$ and $q$.
\end{enumerate}
\end{proposition}

Examples of arrows $A_1\to A_2$ are shown in \cref{arrows sym}.
We emphasize that the conditions for an arrow in \cref{Emily's B arcs} include the condition that $A_1$ is a subarc/subarc pair of $A_2$.
This requirement includes the additional requirement in \cref{subs of pair}, which is crucial for the arrows satisfying condition \eqref{pair pair} in \cref{Emily's B arcs}, as explained in Case 1 of the proof below.

\begin{figure}
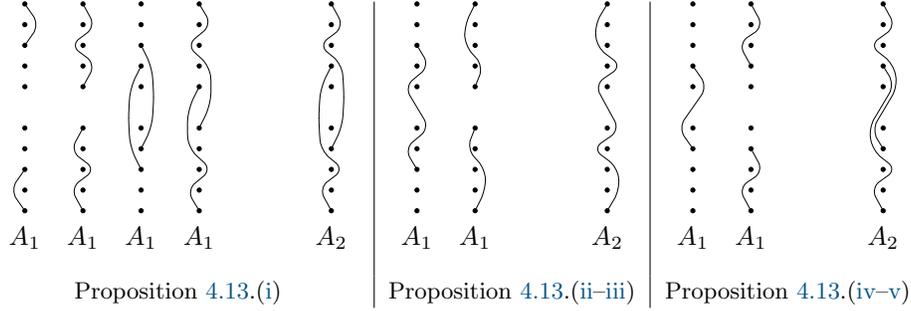

\begin{tabular}{c|c|c}
\begin{tabular}{ccccp{18pt}c}
\includestandalone[height=8em]{figs/b/arrow.S.pair.pair.ord1}
&
\includestandalone[height=8em]{figs/b/arrow.S.pair.pair.ord2}
&
\includestandalone[height=8em]{figs/b/arrow.S.pair.pair.overlap1}
&
\includestandalone[height=8em]{figs/b/arrow.S.pair.pair.overlap2}
&&
\includestandalone[height=8em]{figs/b/arrow.S.pair.pair.big}
\\
$A_1$&$A_1$&$A_1$&$A_1$&&$A_2$
\end{tabular}
&
\begin{tabular}{ccp{18pt}c}
\includestandalone[height=8em]{figs/b/arrow.S.sym.sym.sub}
&\includestandalone[height=8em]{figs/b/arrow.S.sym.non.non1}
&&\includestandalone[height=8em]{figs/b/arrow.S.sym.big}
\\
$A_1$&$A_1$&&$A_2$
\end{tabular}
&
\begin{tabular}{ccp{18pt}c}
\includestandalone[height=8em]{figs/b/arrow.S.over.sym.sym}
&
\includestandalone[height=8em]{figs/b/arrow.S.over.sym.non}
&&
\includestandalone[height=8em]{figs/b/arrow.S.over.sym.big}
\\
$A_1$&$A_1$&&$A_2$ 
\end{tabular}
\\[50pt]
\small\cref{Emily's B arcs}.\eqref{pair pair}
&
\small\cref{Emily's B arcs}.(\ref{sym sym}--\ref{sym non})
&
\small\cref{Emily's B arcs}.(\ref{over sym}--\ref{over non})
\end{tabular}
\caption{Arrows $A_1\to A_2$ among symmetric arcs/pairs}\label{arrows sym}
\end{figure}

\begin{proof}
Let $\Sigma_1$ and $\Sigma_2$ be the shards corresponding to $A_1$ and $A_2$.
\cref{Emily's thm} says that $A_1\to A_2$ if and only if there exists a symmetric arc/pair $A'_1$ (corresponding to a shard $\Sigma_1'$) satisfying certain properties.
In particular, $A_1$ and $A'_1$ must be compatible and there needs to exist a non-basic hyperplane in the rank-two subarrangement containing $H_{\Sigma_1}$ and $H_{\Sigma'_1}$.
In other words, that subarrangement must have more than two hyperplanes.
Each rank-two subarrangement looks like a rank-two parabolic subgroup, so in type B it can have $2$, $3$, or $4$ hyperplanes.

A symmetric arc with endpoints $\pm p$ determines a shard whose hyperplane is orthogonal to $\e_p$.
A symmetric pair such that one arc has endpoints $p$ and $q$ determines a shard whose hyperplane is orthogonal to $\e_q-\e_p$ (recalling the convention that $\e_{-i}$ means $-\e_i$).
If $A_1$ and~$A_1'$ have no endpoints in common, then $H_{\Sigma_1}$ and $H_{\Sigma'_1}$ are orthogonal, and thus, since these hyperplanes are basic in the rank-two subarrangement they determine, that subarrangement has only two hyperplanes. 
Thus we need only consider cases where $A_1$ and~$A_1'$ have endpoints in common.
We break into two cases, where $A_1$ and $A_1'$ determine a rank-two subarrangement with $3$ or $4$ hyperplanes.

\medskip
\noindent
\textbf{Case 1.} 
$H_{\Sigma_1}$ and $H_{\Sigma_1'}$ are the basic hyperplanes in a rank-two subarrangement with~$3$ hyperplanes.
Then $A_1$ is a symmetric pair $\set{\alpha_1,-\alpha_1}$ such that $\alpha_1$ has endpoints $p<q$ and $A_1'$ is a symmetric pair $\set{\alpha_1',-\alpha_1'}$ such that $\alpha_1'$ has endpoints $q<r$ with $q\not\in\set{-p,-r}$.
%The individual arcs $\alpha_1$ and $\alpha_1'$ are thus compatible, so the requirement that $A_1$ and $A_1'$ be compatible is the requirement that $\alpha_1$ and $-\alpha_1'$ are compatible.  
%(Equivalently, $-\alpha_1$ and $\alpha_1'$ are compatible.)
At most one of the pairs $A_1$ and $A_1'$ is overlapping, and if one is, we can assume (up to renaming arcs and points) that $\alpha_1$ or $\alpha'_1$ is to the right of its negative.
The basic hyperplane $H_{\Sigma_1}$ is orthogonal to $\e_q-\e_p$ and the other basic hyperplane $H_{\Sigma'_1}$ is orthogonal to $\e_r-\e_q$.
The unique non-basic hyperplane in the subarrangement is orthogonal to $\e_r-\e_p$.
Thus an arc/arc pair specifies a non-basic hyperplane in the subarrangement if and only if it is a symmetric arc pair  $\set{\alpha_2,-\alpha_2}$ such that $\alpha_2$ has endpoints $p$ and $r$.
%To determine all possible arrows arising from this compatible pair, it remains to determine necessary and sufficient conditions on 
\cref{B shard sym pair} implies that
\begin{multline*}
\Sigma_1\cap\Sigma'_1
=\{\x\in\reals^n:x_p=x_q=x_r,\,\\
x_p\leq x_i\,\forall i\in(R(\alpha_1)\cup R(\alpha_1')),\,\\
x_p\geq x_i\,\forall i\in(L(\alpha_1)\cup L(\alpha_1'))\}.
\end{multline*}
On the other hand,
\[
\Sigma_2=\set{\x\in\reals^n:x_p=x_r,\,x_p\leq x_i\,\forall i\in R(\alpha_2),\,x_p\geq x_i\,\forall i\in L(\alpha_2)}.
\]
Since $\Sigma_1\cap\Sigma'_1$ is in the subspace where $x_p=x_q$, it is contained in $\Sigma_2$ if and only if it is contained in $\set{\x\in\Sigma_2:x_p=x_q}$, which equals
\[
%\set{\x\in\Sigma_2:x_p=x_q}=
\set{\x\in\reals^n:x_p=x_q=x_r,\,x_p\leq x_i\,\forall i\in R(\alpha_2)\setminus\set{q},\,x_p\geq x_i\,\forall i\in L(\alpha_2)\setminus\set{q}}.
\]
Thus $\Sigma_1\cap\Sigma'_1\subseteq \Sigma_2$ if and only if $R(\alpha_2)\setminus\set{q}=R(\alpha_1)\cup R(\alpha'_1)$ (equivalently, if and only if ${L(\alpha_2)\setminus\set{q}=L(\alpha_1)\cup L(\alpha'_1)}$).
This is in turn equivalent to the condition that $\alpha_2$ is obtained by pushing the curve $\alpha_1\cup\alpha'_1$ left or right of the shared endpoint~$q$.

Now \cref{Emily's thm} implies that every arrow among shards that determine rank-two subarrangements with $3$ hyperplanes are of the form $A_1\to A_2$ or $A_1'\to A_2$ for $A_1$, $A_1'$, and $A_2$ as above such that $\alpha_2$ is obtained by pushing the curve $\alpha_1\cup\alpha'_1$ left or right of~$q$.
Given $A_1$ and $A_2$, the existence of $A'_1$ such that $\alpha_2$ is obtained by pushing $\alpha_1\cup\alpha'_1$ left or right of~$q$ is exactly the condition that $A_1$ is a subarc pair of $A_2$, \emph{except for the additional requirement on subarc pairs of symmetric arc pairs} in \cref{subs of pair}.
Similarly, given $A_1'$ and $A_2$, the existence of an appropriate $A_1$ is exactly that $A_1'$ is a subarc pair, without the additional requirement.

Furthermore, \cref{Emily's thm} says that if $\alpha_2$ is obtained by pushing $\alpha_1\cup\alpha'_1$ left or right of~$q$, then these arrows $A_1\to A_2$ or $A_1'\to A_2$ exist if and only if $A_1$ and $A_1'$ are compatible.
We will show that, under the condition that $\alpha_2$ is obtained by pushing $\alpha_1\cup\alpha'_1$ left or right of~$q$, the compatibility of $A_1$ and $A_1'$ is equivalent to the additional requirement.  

Compatibility of $A_1$ and $A_1'$ means that all four arcs $\alpha_1$, $-\alpha_1$, $\alpha'_1$, and $-\alpha'_1$ can be in the same noncrossing arc diagram.
Recall also that if one of the pairs $\set{\alpha_1,-\alpha_1}$ or $\set{\alpha'_1,-\alpha'_1}$ is overlapping, then $\alpha_1$ or $\alpha_1'$ is on the right.  
For each pair which is nonoverlapping, $\alpha_1$ or $\alpha_1'$ is above its opposite.  

Suppose $A_1$ and $A_1'$ are compatible. 
Then the union $\alpha_1\cup\alpha'_1$ is either above or right of its opposite $(-\alpha_1)\cup(-\alpha'_1)$.
Since $\alpha_2$ is obtained by pushing $\alpha_1\cup\alpha'_1$ left or right of~$q$, then either $\set{\alpha_2,-\alpha_2}$ is nonoverlapping and $\alpha_2$ is above $-\alpha_2$ or $\set{\alpha_2,-\alpha_2}$ is overlapping and $\alpha_2$ is right of $-\alpha_2$.

Now, suppose $A_1$ and $A_1'$ are not compatible. 
The individual arcs $\alpha_1$ and $\alpha_1'$ are compatible in any case, so the failure of compatibility of $A_1$ and $A_1'$ implies that $\alpha_1$ and $-\alpha_1'$ cross each other. 
Thus the embedding of $\alpha_2$ obtained by pushing $\alpha_1\cup\alpha'_1$ left or right of~$q$ also crosses its antipodal opposite.  
Since by supposition $\set{\alpha_2,-\alpha_2}$ is a symmetric arc pair, there is an embedding of $\alpha_2$ that does not cross its opposite.
This embedding must therefore have $\alpha_2$ left of $-\alpha_2$, as in \cref{fig:subarc fail}.

We have shown that arrows of the form $A_1\to A_2$ and $A_1'\to A_2$ arising from Case 1 are precisely the arrows described in condition \eqref{pair pair}.

\medskip
\noindent
\textbf{Case 2.} 
$H_{\Sigma_1}$ and $H_{\Sigma_1'}$ are the basic hyperplanes in a rank-two subarrangement with~$4$ hyperplanes. 
Then (up to swapping $A_1$ and $A_1'$) $A_1$ is a symmetric arc $\alpha_1$ with endpoints $\pm p$ for $p>0$ and $A_1'$ is a non-overlapping arc pair $\set{\alpha_1',-\alpha_1'}$ such that $\alpha_1'$ has endpoints $p$ and $q$ with $p<q$.
Any such $A_1$ and $A_1'$ are compatible, with no additional requirements needed.

The basic hyperplanes in the subarrangement are orthogonal to $\e_p$ and $\e_q-\e_p$.
The two non-basic hyperplanes are $\e_q$ and $\e_q+\e_p$.
Thus an arc/arc-pair $A_2$ specifies a non-basic hyperplane in the subarrangement if and only if it is a symmetric arc with endpoints $\pm q$ or an overlapping symmetric pair one of whose arcs has endpoints $-p$ and $q$.
To determine all possible arrows arising from this compatible pair, it remains to determine necessary and sufficient conditions on $A_2$ so that the corresponding shard has $\Sigma_1\cap\Sigma'_1\subseteq\Sigma_2$.

\cref{B shard sym arc,B shard sym pair} combine to say that 
\[
\Sigma_1\cap\Sigma'_1
=\set{\x\in\reals^n:x_p=x_q=0,\,0\leq x_i\forall i\in(R(\alpha_1)\cup R(\alpha_1')),\,0\geq x_i\forall i\in L(\alpha_1')}.
\]
Since $x_{-i}=-x_i$, we can rewrite this as 
\begin{equation}\label{Sig1 Sig1p case 2}
\Sigma_1\cap\Sigma'_1
=\set{\x\in\reals^n:x_p=x_q=0,\,0\leq x_i\forall i\in(R(\alpha_1)\cup R(\alpha_1')\cup(-L(\alpha_1')))}.
\end{equation}
Since $\Sigma_1\cap\Sigma'_1$ is in the subspace where $x_p=0$, it is contained in $\Sigma_2$ if and only if it is contained in $\set{\x\in\Sigma_2:x_p=0}$.

\medskip
\noindent
\textit{Case 2a.}
If $A_2$ is a symmetric arc $\alpha_2$ with endpoints $\pm q$, then \cref{B shard sym arc} says that 
\[\set{\x\in\Sigma_2:x_p=0}=\set{\x\in\reals^n:x_p=x_q=0,\,0\leq x_i\forall i\in R(\alpha_2)}.\]

If $A_1$ is a subarc of $A_2$ as in condition~\eqref{sym sym} and $A_1'$ is a subarc pair of $A_2$ as in condition~\eqref{sym non}, then in particular ${R(\alpha_2)\setminus\set{\pm p}}=R(\alpha_1)\cup R(\alpha_1')\cup(-L(\alpha_1'))$, so ${\Sigma_1\cap\Sigma'_1=\set{\x\in\Sigma_2:x_p=0}}$.

Conversely, if $\Sigma_1\cap\Sigma'_1\subseteq\Sigma_2$, then for every $i\in R(\alpha_2)$, the inequality $0\le x_i$ holds in $\Sigma_1\cap\Sigma'_1$.
Thus $R(\alpha_2)\setminus\set{\pm p}=R(\alpha_1)\cup R(\alpha_1')\cup(-L(\alpha_1'))$, so $A_1$ is a subarc of $A_2$ and $A_1'$ is a subarc pair of $A_2$.

We have shown that arrows of the form $A_1\to A_2$ and $A_1'\to A_2$ in Subcase 2a are precisely the arrows described in conditions \eqref{sym sym} and \eqref{sym non}.

\medskip
\noindent
\textit{Case 2b.}
If $A_2$ is an overlapping symmetric pair $\set{\alpha_2,-\alpha_2}$ such that $\alpha_2$ has endpoints $-p$ and $q$, then we consider \cref{B shard sym pair} in two further cases, given by whether $\alpha_2$ is to the right or left of $-\alpha_2$.
If $\alpha_2$ is to the right of $-\alpha_2$, then \cref{B shard sym pair} implies that 
%\[\Sigma_2=\set{\x\in\reals^n:x_{-p}=x_q,\,x_{-p}\leq x_i\forall i\in R(\alpha_2),\,x_{-p}\geq x_i\forall i\in L(\alpha_2)},\text{ so}\] 
\[\set{\x\in\Sigma_2:x_p=0}=\set{\x\in\reals^n:x_p=x_q=0,\,0\leq x_i\forall i\in R(
\alpha_2),\,0\geq x_i\forall i\in L(\alpha_2)}.\]
If $\alpha_2$ is to the left of $-\alpha_2$, then since $-L(\alpha_2)$ is the set of points to the \emph{right} of $-\alpha_2$ and $-R(\alpha_2)$ is the set of points to the \emph{left} of $-\alpha_2$, \cref{B shard sym pair} implies that 
\begin{multline*}
\set{\x\in\Sigma_2:x_p=0}
\\=\set{\x\in\reals^n:x_p=x_q=0,\,0\leq x_i\forall i\in -L(\alpha_2),\,0\geq x_i\forall i\in -R(\alpha_2)}.
\end{multline*}
In either case, 
\begin{multline}\label{Sig2 case 2}
\set{\x\in\Sigma_2:x_p=0}\\=\set{\x\in\reals^n:x_p=x_q=0,\,0\leq x_i\forall i\in (R(\alpha_2)\cup(-L(\alpha_2)))}.
\end{multline}

If $A_1$ is a subarc of $A_2$ as in condition~\eqref{over sym} and $A_1'$ is a subarc pair of $A_2$ as in condition~\eqref{over non}, then no points are between $\alpha_2$ and $-\alpha_2$, so $R(\alpha_2)\cup(-L(\alpha_2))$ contains no pairs $\pm i$.
Furthermore, $R(\alpha_2)\cap(-p,p)=R(\alpha_1)$, and ${R(\alpha_2)\cap(p,q)=R(\alpha_1')}$, and $L(\alpha_2)\cap(p,q)=L(\alpha_1')$. 
Comparing \eqref{Sig2 case 2} with \eqref{Sig1 Sig1p case 2}, we see that $\Sigma_1\cap\Sigma'_1$ equals $\set{\x\in\Sigma_2:x_p=0}$.

Conversely, if $\Sigma_1\cap\Sigma_1'\subseteq\set{\x\in\Sigma_2:x_p=0}$, then by \eqref{Sig1 Sig1p case 2} and \eqref{Sig2 case 2}, we know that ${(R(\alpha_2)\cup(-L(\alpha_2)))\subseteq(R(\alpha_1)\cup R(\alpha_1')\cup(-L(\alpha_1')))}$.
In this case, no point can be between $\alpha_2$ and $-\alpha_2$, because $R(\alpha_1)$ contains no pairs $\pm i$.
Thus $R(\alpha_2)$ and $-L(\alpha_2)$ coincide along the interval $(-p,p)$.
Therefore ${R(\alpha_2)\setminus\set{p}\subseteq R(\alpha_1)\cup R(\alpha_1')}$ and ${L(\alpha_2)\setminus\set{p}\subseteq (-R(\alpha_1))\cup L(\alpha_1')}=L(\alpha_1)\cup L(\alpha_1')$, so $A_1$ is a subarc of $A_2$ and $A_1'$ is a subarc pair of $A_2$.

We have shown that arrows of the form $A_1\to A_2$ and $A_1'\to A_2$ in Subcase 2b are precisely the arrows described in conditions \eqref{over sym} and \eqref{over non}.
\end{proof}

\begin{proof}[Proof of \cref{arc forcing B}]
Both directions of the theorem use \cref{Emily's B arcs}.
Let $A_1$ be the symmetric arc/pair corresponding to $j_1$ and $A_2$ be the symmetric arc/pair corresponding to $j_2$.
We continue to use the notation $A_1\to A_2$ to mean $j_1\to j_2$.

If $j_i$ forces $j_2$, then there is a sequence of arrows from $j_1$ to $j_2$.
Thus \cref{Emily's B arcs,B sub trans} imply that $A_1$ is a subarc/subarc pair of $A_2$.
%implies in particular that there is a sequence $A_1=A'_1,A'_2,\ldots,A'_k=A_2$ of symmetric arcs/pairs such that $A'_i$ is a subarc of $A'_{i+1}$ for all $i=1,\ldots,k-1$.

Conversely, suppose that $A_1$ is a subarc/subarc pair of $A_2$, with $A_1\neq A_2$.
We will show that, in every case, there is a sequence of one, two, or three arrows from $j_1$ to $j_2$.
For the purposes of this proof, we will refer to the kinds of arrows described in \cref{Emily's B arcs} as ``type (i)'', etc.

\medskip
\noindent
\textbf{Case 1.}
$A_2$ is a symmetric arc.

\medskip
\noindent
\textit{Case 1a.}
$A_1$ is also a symmetric arc.
In this case, $A_1\to A_2$ by an arrow of type~\eqref{sym sym}

\medskip
\noindent
\textit{Case 1b.}
$A_1$ is a non-overlapping symmetric pair.
In this case, suppose that the endpoints of $A_2$ are $\pm p$ with $p>0$ and the endpoints of one arc of $A_1$ are $p'$ and $q'$ with $0<p'<q'\le p$.
Then there is an arrow of type~\eqref{sym non} from $A_1$ to a symmetric arc $A'$ with endpoints $\pm q'$ that is a subarc of $A_2$, and (unless $A'=A_2$) an arrow of type~\eqref{sym sym} from $A'$ to $A_2$, as illustrated in the left picture of \cref{fig:arc forcing B}.
\begin{figure}
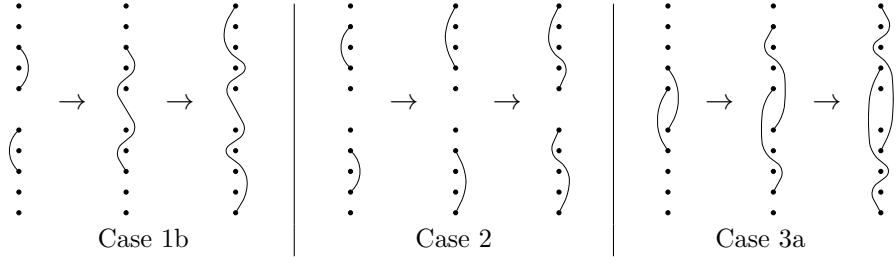

\begin{tabular}{c|c|c}
    \includestandalone[height=8em]{figs/b/subarc.S.sym.non}
    \ \ 
    \raisebox{4em}{\large $\to$}
    \ \ 
    \includestandalone[height=8em]{figs/b/arrow.S.sym.sym.sub}
    \ \ 
    \raisebox{4em}{\large $\to$}
    \ \ 
    \includestandalone[height=8em]{figs/b/arrow.S.sym.big}
    \hspace*{8pt}
&
    \hspace*{8pt}
    \includestandalone[height=8em]{figs/b/subarc.S.non.sub}
    \ \ 
    \raisebox{4em}{\large $\to$}
    \ \ 
    \includestandalone[height=8em]{figs/b/subarc.S.non.mid}
    \ \ 
    \raisebox{4em}{\large $\to$} 
    \ \ 
    \includestandalone[height=8em]{figs/b/subarc.S.non.big}
    \hspace*{8pt}
&
    \hspace*{8pt}
    \includestandalone[height=8em]{figs/b/subarc.S.over.sub}
    \ \ 
    \raisebox{4em}{\large $\to$}
    \ \ 
    \includestandalone[height=8em]{figs/b/subarc.S.over.mid}
    \ \ 
    \raisebox{4em}{\large $\to$}
    \ \ 
    \includestandalone[height=8em]{figs/b/arrow.S.pair.pair.big}
\\
Case 1b&Case 2&Case 3a
\end{tabular}
\caption{Some illustrations of the proof of \cref{arc forcing B}}
\label{fig:arc forcing B}
\end{figure}

\medskip
\noindent
\textbf{Case 2.} 
$A_2$ is a nonoverlapping symmetric pair.
In this case, $A_1$ is also a nonoverlapping symmetric pair.
If the outer endpoints of $A_1$ and $A_2$ are not the same, then there is an arrow of type \eqref{pair pair} from $A_1$ to an arc pair $A'$ whose outer endpoints agree with those of $A_2$; otherwise, let $A'=A_1$. 
If the inner endpoints of $A'$ and $A_2$ agree then $A'=A_2$.
If not, then there is another arrow of type \eqref{pair pair} from $A'$ to $A_2$, as illustrated in the center picture of \cref{fig:arc forcing B}.

\medskip
\noindent
\textbf{Case 3.}
$A_2=\set{\alpha_2,-\alpha_2}$ is an overlapping symmetric pair.  
Write $p$ and $q$ with $p<q$ for the endpoints of $\alpha_2$.

\medskip
\noindent
\textit{Case 3a.}
$A_1=\set{\alpha_1,-\alpha_1}$ is also an overlapping symmetric pair.
In this case, assume that $\alpha_2$ is right of $-\alpha_2$, that $\alpha_1$ is right of~$-\alpha_1$, and that $\alpha_1$ has endpoints $p'$ and $q'$ with $p\le p'<0<q'\le q$.
If $p<p'<0<q'<q$, then we find a sequence of two arrows of type~\eqref{pair pair} from $A_1$ to $A_2$. 
If $q'=q$ and/or $p=p'$, then one or both of these arrows is replaced by equality.
If $-q>p$, then also $-q'>p$, so there is a type~\eqref{pair pair} arrow from $A_1$ to a subarc pair $A'=\set{\alpha',-\alpha'}$ of $A_2$ such that $\alpha'$ has endpoints $p$ and $q'$ and an arrow of type~\eqref{pair pair} from $A'$ to $A_2$, as illustrated in the right picture of \cref{fig:arc forcing B}.
If $p>-q$, then also $p'>-q$, so there is a type~\eqref{pair pair} arrow from $A_1$ to a subarc pair $A'=\set{\alpha',-\alpha'}$ of $A_2$ such that $\alpha'$ has endpoints $p'$ and~$q$ and an arrow of type~\eqref{pair pair} from $A'$ to $A_2$.
(Separating into two cases $-q>p$ and $p>-q$ is necessary.
For example, when $-q>p$, it is possible that $p'=-q$, so that there is no subarc pair $A'=\set{\alpha',-\alpha'}$ of $A_2$ such that $\alpha'$ has endpoints $p'$ and~$q$.
This is the case illustrated in the right picture of \cref{fig:arc forcing B}.)

\medskip
\noindent
\textit{Case 3b.}
$A_1=\set{\alpha_1,-\alpha_1}$ is a non-overlapping symmetric pair.
In this case, assume that $\alpha_1$ is a subarc of $\alpha_2$ (making no assumption about which of $\alpha_2$ or $-\alpha_2$ is to the right).
Write $p'$ and $q'$ for the endpoints of $\alpha_1$, this time with $p<0<p'<q'\leq q$.
If $p'\neq-p$, then there is a subarc pair $A'$ of $A_2$ with endpoints $p'$ and $q$ and an arrow $A'\to A_2$ of type \eqref{pair pair}. 
If $p'=-p$, then there exists a subarc pair $A'=\{\alpha',-\alpha'\}$ of $A_2$ such that $\alpha'$ has endpoints $1$ and $q$ and an arrow from $A'$ to $A_2$. 
The arrow is of type \eqref{pair pair} if $p<-1$ or \eqref{over non} if $p=-1$.
Either way, $A_1$ is a subarc of $A'$ and by Case~2, there is a (possibly empty) sequence of arrows from $A_1$ to $A'$.

\medskip
\noindent
\textit{Case 3c.}
$A_1$ is a symmetric arc $\alpha_1$.
Write $p'$ and $-p'$ with ${p\leq p'<0<-p'\leq q}$ for the endpoints of $\alpha_1$. 
If $p'=p$ or $-p'=q$, then there is an arrow of type \eqref{over sym} from $A_1$ to $A_2$, so it remains to consider the case where $p<p'<0<-p'<q$. 
Without loss of generality (up to swapping $\pm \alpha_2$), we may as well assume that $p>-q$. 
Therefore also $p'>-q$, so there is a type \eqref{over sym} arrow from $A_1$ to a subarc pair $A'=\{\alpha',-\alpha'\}$ of $A_2$ such that $\alpha'$ has endpoints $p'$ and $q$ and an arrow from $A'$ to $A_2$ of type \eqref{pair pair}.
\end{proof}

%\subsection{An orbifold model}\label{B orb sec}
\section{Orbifold noncrossing arc diagrams}\label{B orb sec}
We now take the centrally symmetric model for Coxeter groups of type $B_n$ and pass to a quotient modulo the central symmetry, obtaining what we call an orbifold model.
Most simply---but not very ``drawably''---the orbifold model lives in the quotient space where each point of the plane is identified with its antipodal opposite.
The quotient map takes a point $x$ in the plane to $\set{\pm x}$.  

To make a more ``drawable'' model, we consider the same space as a different quotient of the plane:
We cut the plane in half with a horizontal line through the origin.  
Points strictly above the horizontal line are not identified with any other points.
Each point on or below the horizontal line is identified with all other points on or below the line at the same distance from the origin.
Thus, each ``point'' in the quotient is either a point strictly above the horizontal line (a \newword{point in the upper halfplane}), the origin, or a semicircle below the line with endpoints on the line (or a degenerate semicircle consisting only of the origin, called a \newword{semicircle point}).
We will refer to this quotient as the \newword{orbifold plane}.

We are not interested in the natural quotient map associated to the orbifold plane.
Rather, we are interested in the map $\phi$ defined as follows:
If $x$ is not on the horizontal line, then $\phi(x)$ is the point in $\set{\pm x}$ that is above the horizontal line.
If $x$ is on the horizontal line, then $\phi(x)$ is the semicircle of points of length $|x|$ on or below the horizontal line (or $\phi(0)$ is the degenerate semicircle at the origin).

%Suppose $\alpha$ is a symmetric arc, in the sense of \cref{B sym sec}, again parametrized as a function from $[0,1]$ with $\alpha(t)=-\alpha(1-t)$.
%The function sending $t$ to $\phi(\alpha(t))$ is a curve $\phi(\alpha)$. 
%The image of $\phi(\alpha)$ coincides with the image of the restriction of $\phi(\alpha)$ to $t\in[0,\frac12]$.
%Similarly, suppose $\set{\alpha,-\alpha}$ is a centrally symmetric pair of compatible arcs, in the sense of \cref{B sym sec}.
%Then the curve $\phi(\alpha)$, which sends $t$ to $\phi(\alpha(t))$, coincides with the curve $\phi(-\alpha)$, which sends $t$ to $\phi(-\alpha(t))$.

A symmetric arc $\alpha$ or pair $\set{\alpha,-\alpha}$ in the centrally symmetric model can be uniquely recovered from $\phi(\alpha)$.
Thus also symmetric arc diagrams and the various results and constructions in \cref{B sym sec} can be recovered from their images under $\phi$.
The goal now is to define arcs, arc diagrams, a bijection to $B_n$, etc. in the orbifold plane so that the results of \cref{B sym sec} translate to results in the orbifold plane via the map $\phi$.
With the right definitions, these new results are simply ``translations'' into a new setting, and will not require new proofs.
Collectively, we will refer to these constructions and results as the \newword{orbifold model} for $B_n$.

The orbifold model for $B_n$ starts with $n$ distinct points on a vertical line containing the origin, with each point strictly above the origin.
The origin itself is called the \newword{orbifold point}, and is marked with an ``\orb''.
(We will sometimes also refer, in prose, to the orbifold point as ``\orb''.)
The $n$ points above the origin are identified with the numbers $1,2,\ldots, n$, in order, with $n$ at the top, and are called \newword{numbered points}

A \newword{type-B arc} (or in context simply an \newword{arc}) is a curve in the orbifold plane with each endpoint at a numbered point or at \orb (the origin), satisfying one of the following three descriptions: 
\begin{itemize}
\item 
An \newword{ordinary arc} is an arc on the points $1,\dots,n$ satisfying the same rules as an arc in type A (\cref{A sec}).
\item
An \newword{orbifold arc} is an arc with one endpoint at \orb and the other at a numbered point $p$, moving monotone downwards from $p$ to \orb without touching any other numbered point, passing to the left or right of any numbered points below $p$.
\item 
A \newword{long arc} is an arc $\alpha$ containing exactly one semicircle point, which is not degenerate (i.e. not equal to the orbifold point).
The \newword{left piece} of $\alpha$ moves monotone downward from the \newword{left endpoint} of $\alpha$ and hits the semicircle point left of \orb, passing left or right of numbered points between.
The \newword{right piece} moves monotone downward from the \newword{right endpoint} and hits the semicircle point right of \orb, passing left or right of numbered points between.
The left and right pieces do not intersect and their endpoints do not coincide.
\end{itemize}

An ordinary arc or orbifold arc $\alpha$ is specified combinatorially by its endpoints and the set $R(\alpha)$ of numbered points to its right, or equivalently by its endpoints and the set $L(\alpha)$ of numbered points to its left.
A long arc $\alpha$ is specified combinatorially by its endpoints (right and left), the set $R(\alpha)$ of points right of its right piece \emph{and} the set $L(\alpha)$ of points left of its left piece.
We emphasize that for long arcs (in contrast to the situation for ordinary and orbifold arcs), the sets $R(\alpha)$ and $L(\alpha)$ do not completely determine each other, because points not in $R(\alpha)$ could be in $L(\alpha)$ or between the left and right pieces.
However, $R(\alpha)$ and $L(\alpha)$ must be disjoint, and both sets are necessary to determine $\alpha$ combinatorially. 
We also emphasize that to determine $\alpha$, it is necessary to specify which endpoint is right and which is left.
(For example, for each pair of distinct endpoints, there are two long arcs $\alpha$ with $R(\alpha)=\emptyset$ and $L(\alpha)=\emptyset$, depending on which endpoint is right and which is left.)

A \newword{(type-B) noncrossing arc diagram (on $n$ points)} is a collection of type-B arcs on points \orb and $1,\ldots n$ that don't intersect, except possibly at their endpoints, with no two arcs sharing an endpoint from which they both go down or both go up.
Again, we consider arcs and noncrossing arc diagrams up to combinatorial equivalence.
All type-$B_2$ and type-$B_3$ noncrossing arc diagrams are shown in \cref{fig:b2orbncads,fig:b3allncads}.
For easy comparison, the diagrams in \cref{fig:b2symncads,fig:b2orbncads} appear in the same order.

\begin{figure}
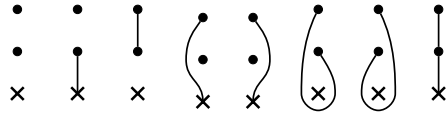

    \centering
\includestandalone[scale=0.8]{figs/b/b2.orb.allncads}
    \caption{Type-B noncrossing arc diagrams for $n=2$}
    \label{fig:b2orbncads}
\end{figure}

\begin{figure}[t]
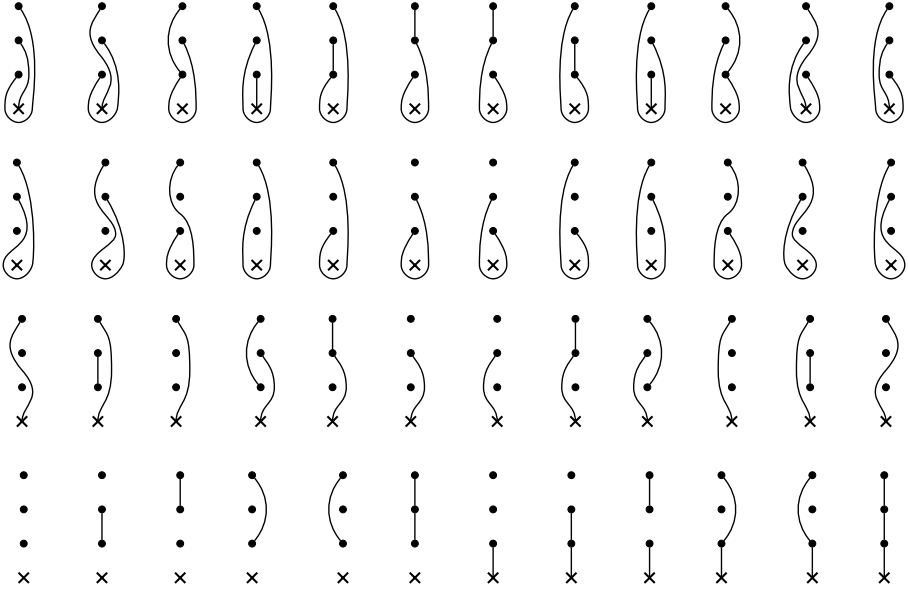

    \centering
    \includestandalone[scale=0.65]{figs/b/b3.allncads.4rows}
    \caption{Type-B noncrossing arc diagrams for $n=3$}
    \label{fig:b3allncads}
\end{figure}

The map $\phi$ is a bijection from the set of centrally symmetric arcs and centrally symmetric pairs of arcs to the set of type-B arcs, and also induces a bijection on combinatorial equivalence classes.
Nonoverlapping symmetric pairs, symmetric arcs, and overlapping symmetric pairs map respectively to ordinary, orbifold, and long arcs.
Furthermore, $\phi$ induces a bijection from centrally symmetric noncrossing arc diagrams to type-B noncrossing arc diagrams.
We write~$\delta^\circ$ for the map $\phi\circ\delta$, where $\delta$ is understood as in \cref{B sym sec} to map signed permutations to centrally symmetric noncrossing arc diagrams.
The following theorem is a restatement of \cref{main B}.
\begin{theorem}\label{main B orb}
The map $\delta^\circ$ is a bijection from $B_n$ to the set of type-B noncrossing arc diagrams on $n$ points.
\end{theorem}
%(In the theorem as in the definition above, diagrams are constructed on $n$ points, not including the orbifold point at the origin.)

We now give a direct description of the bijection $\delta^\circ$ from $B_n$ to the set of type-B noncrossing arc diagrams on $n$ points, as illustrated in \cref{fig:b.map.orb}.
%Given $\pi=\pi_1\cdots \pi_n\in B_n$, write each entry $\pi_i$ at the point $(i,\pi_i)$ in the plane. 
Given $\pi\in B_n$ with one-line notation $\pi_1\cdots \pi_n$, write each entry $\pi_i$ at the point $(i,\pi_i)$ in the plane, for $i=1,\ldots,n$. 
For every $i$ such that $\pi_i>\pi_{i+1}$, draw a straight line segment from $\pi_i$ to $\pi_{i+1}$. 
Additionally, if $\pi_1<0$, draw a line segment from the origin to $\pi_1$.
These line segments become arcs: 
%First, we move the numbers $\pi_1,\ldots,\pi_n$ (with $\{|\pi_1|,\ldots,|\pi_n|\}=\{1,\ldots,n\}$) horizontally to put them into a single vertical line, allowing the line segments to curve, so that they don't pass through any of the numbers. 
First, we move the numbers $\pi_1,\ldots,\pi_n$ horizontally to put them into a single vertical line, allowing the line segments to curve, so that they don't pass through any of the numbers or one another. 
Then, we rotate the negative numbers 180 degrees clockwise, allowing the (already curved) line segments connecting positive numbers to negative numbers to stretch around the origin. 
We remove the negative signs on the numbers.
Since $\set{|\pi_1|,\ldots,|\pi_n|}=\set{1,\ldots,n}$, the numbers are now $1,\ldots,n$ from bottom to top.
%We define $\delta^\circ(\pi)$ to be the resulting type-B noncrossing arc diagram. 
The resulting type-B noncrossing arc diagram is~$\delta^\circ(\pi)$.

\begin{figure}
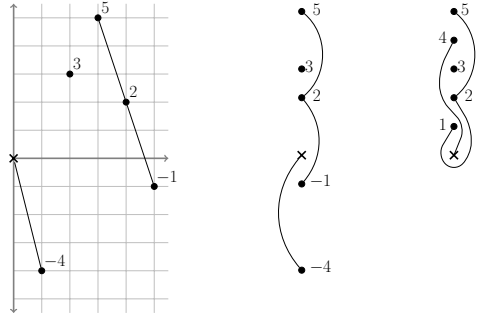

\includestandalone[height=12em]{figs/b/orb.draw}
\qquad \quad   
\includestandalone[height=12em]{figs/b/orb.smoosh}
\qquad \quad 
\includestandalone[height=12em]{figs/b/orb.wrap}
\caption{The map $\delta^\circ$ applied to $(-4)352(-1)$}
\label{fig:b.map.orb}
\end{figure}

We continue to translate the results of \cref{B sym sec} into the language of the orbifold model.
We next describe the bijection from arcs to join-irreducible elements of $B_n$.  
We continue the conventions from \cref{B sym sec} for describing join-irreducible signed permutations.

A join-irreducible signed permutation $\pi$ with a single descent $\pi_{-1}>\pi_1$ corresponds to an orbifold arc $\alpha$.
If $\alpha$ has upper endpoint $p$, then $\pi$ is
\[(-p)\,(-L(\alpha))\,R(\alpha)\,(p+1)\cdots n.\]

A join-irreducible signed permutation $\pi$ with two symmetric descents $\pi_i>\pi_{i+1}$ and $\pi_{-i-1}>\pi_{-i}$ such that $\pi_i$ and $\pi_{i+1}$ have the same sign corresponds to an ordinary arc $\alpha$.
(If $\pi_i$ and $\pi_{i+1}$ are both negative, there must be an additional inversion between $\pi_i$ and the entry before it, either $\pi_{i-1}$ or $\pi_{-1}$.) 
If $\alpha$ has endpoints $p$ and $q$ with $0<p<q$, then $\pi$ is
\[1\cdots(p-1)\,L(\alpha)\,q\,p\,R(\alpha)\,(q+1)\cdots n.\]

A join-irreducible element $\pi$ with two symmetric descents each involving a positive and a negative number corresponds to a long arc $\alpha$.
Suppose $\alpha$ has left endpoint $p$ and right endpoint $q$ (noting that as we translate this description from \cref{B sym sec}, the $p$ here is the $-p$ there).
Necessarily, $p\neq q$.
Write $B(\alpha)$ (suggesting ``between'') to denote the set of points that are left of the right piece of $\alpha$ and right of the left piece.
Thus $B(\alpha)$ is a subset of $(0,\min(p,q))$.
If $p<q$, then $\pi$ is 
\[B(\alpha)\,[(p,q)\setminus R(\alpha)]\,q\,(-p)\,(-L(\alpha))\,R(\alpha)(q+1)\cdots n.\]
If $q<p$, then $\pi$ is
\[B(\alpha)\,q\,(-p)\,(-L(\alpha))\,R(\alpha)\,[(q,p)\setminus L(\alpha)]\,(p+1)\cdots n.\]

The following is a restatement of \cref{B CJR} in the orbifold model.

\begin{theorem}\label{B orb CJR}
Given $\pi\in B_n$, the canonical join representation of $\pi$ is the set of join-irreducible elements corresponding to the arcs in $\delta^\circ(\pi)$.
\end{theorem}

Two type-B arcs are \newword{compatible} if and only if they do not intersect, except possibly at numbered endpoints, and don't share the same top endpoint or the same bottom endpoint. 
In other words, they are compatible if and only if they can appear together in a type-B noncrossing arc diagram.
As a consequence of \cref{main B orb,B orb CJR}, two join-irreducible elements of $B_n$ are compatible (can appear together in a canonical join representation) if and only if the corresponding type-B arcs are compatible.

The following three propositions are restatements of \cref{B shard sym arc,B shard sym pair}.
(\cref{B shard A,B shard long} break \cref{B shard sym pair} into two cases.)

\begin{proposition}\label{B shard to orb}
Suppose $\alpha$ is an orbifold arc with upper endpoint~$p$.
The shard associated to $\alpha$ is 
\[\set{\x\in\reals^n:x_p=0\text{ and }0\leq x_i\,\forall i\in R(\alpha)\text{ and }0\geq x_i\,\forall i\in L(\alpha)}.\]
\end{proposition}

\begin{proposition}\label{B shard A}  
Suppose $\alpha$ is an ordinary arc with endpoints $p$ and $q$.
The shard associated to~$\alpha$ is 
\[\set{\x\in\reals^n:x_p=x_q\text{ and }x_p\leq x_i\,\forall i\in R(\alpha)\text{ and }x_p\geq x_i\,\forall i\in L(\alpha)}.\]
\end{proposition}

\begin{proposition}\label{B shard long}  
Suppose $\alpha$ is a long arc with left endpoint $p$ and right endpoint~$q$.
The shard associated to $\alpha$ is 
\begin{multline*}
\bigl\{\x\in\reals^n:x_q=-x_p\text{ and }x_q\leq x_i\,\forall i\in (R(\alpha)\cup(-L(\alpha)))\\
\text{ and }x_q\geq x_i\,\forall i\in ((-p,q)\setminus( R(\alpha)\cup(-L(\alpha)))\bigr\}.
\end{multline*}
\end{proposition}

We now turn to translating the definitions of subarcs and superarcs, as well as \cref{arc forcing B}, to the orbifold model, using the map $\phi$.
There is a case-free way to define subarcs of a type-B arc $\alpha$:
Informally, we can cut $\alpha$ shorter at one or both of its endpoints and reattach the shortened curve to a numbered point.  
If the result is a type-B arc (or can be easily modified to become an orbifold arc as made precise below), then it is a subarc of $\alpha$.
Together, \cref{subs of ord/orb,subs of long}, below, correspond to \cref{subs of sym,subs of pair}, but the split into two definitions is not the same here as in \cref{B sym sec}.

\begin{figure}
\centering
\begin{tabular}{c|c}
\begin{tabular}{cccp{18pt}c}
\includestandalone[height=5em]{figs/b/arrow.O.orb.orb.sub}
&\includestandalone[height=5em]{figs/b/subarc.O.orb.non}
&\includestandalone[height=5em]{figs/b/subarc.O.orb.non1}
&&\includestandalone[height=5em]{figs/b/arrow.O.orb.big}
\\
$\alpha'$&$\alpha'$&$\alpha'$&&$\alpha$
\end{tabular}
&
\begin{tabular}{cccp{18pt}c}
\includestandalone[height=5em]{figs/b/arrow.O.long.orb.orb}
&
\includestandalone[height=5em]{figs/b/subarc.O.long.mid}
&
\includestandalone[height=5em]{figs/b/arrow.O.long.orb.A}
&&
\includestandalone[height=5em]{figs/b/arrow.O.long.big}
\\
$\alpha'$&$\alpha'$&$\alpha'$&&$\alpha$
\end{tabular}
\\[35pt]
\small\cref{subs of ord/orb}
&
\small\cref{subs of long}
\end{tabular}
\caption{Subarcs}
\label{subarcs orb}
\end{figure}

\begin{definition}[Subarcs of ordinary and orbifold arcs]\label{subs of ord/orb}
Suppose $\alpha$ is an ordinary or orbifold arc with endpoints $p<q$, with \orb$=0$ for the purposes of this definition.
A \newword{subarc} of $\alpha$ is an arc $\alpha'$ with endpoints $p'$ and $q'$ having $p\le p'<q'\le q$ and $R(\alpha')=R(\alpha)\cap(p',q')$.
The arc $\alpha'$ is ordinary unless $p'=0$, in which case $\alpha'$ is an orbifold arc.
\end{definition}

\begin{definition}[Subarcs of long arcs]\label{subs of long}
Suppose $\alpha$ is a long arc with left endpoint~$p$ and right endpoint $q$.
A \newword{subarc} of $\alpha$ can be ordinary, orbifold, or long.
\begin{itemize}
\item 
An ordinary arc $\alpha'$ with endpoints $p'<q'$ is a subarc of $\alpha$ if and only if one of the following occurs:
either $q'\le p$ and $L(\alpha')=L(\alpha)\cap(p',q')$ or $q'\le q$ and $R(\alpha')=R(\alpha)\cap(p',q')$.
%(Loosely, $\alpha'$ is a subarc of the left piece of $\alpha$ or of the right piece of $\alpha$.)
\item
An orbifold arc $\alpha'$ with upper endpoint $p'$ is a subarc of $\alpha$ if and only if $p'\le\min(p,q)$, $L(\alpha')=L(\alpha)\cap(0,p')$, and $R(\alpha')=R(\alpha)\cap(0,p')$.
%(Loosely, $\alpha'$ is a subarc of both the left piece of $\alpha$ and the right piece of $\alpha$.)
\item 
A long arc $\alpha'$ with left endpoint $p'$ and right endpoint $q'$ is a subarc of $\alpha$ if and only if $p'\le p$, $q'\le q$, $L(\alpha')=L(\alpha)\cap(0,p')$, $R(\alpha')=R(\alpha)\cap(0,q')$, and the following additional conditions hold:
$p'\not\in R(\alpha)$ and $q'\not\in L(\alpha)$.
\end{itemize}
\end{definition}

\begin{remark}\label{orb explicit}
To better understand long subarcs of long arcs in \cref{subs of long}, we discuss an explicit construction of these subarcs.
Compare similar constructions in \cref{A explicit,sym explicit}.
Take $h$ to be a height function on the orbifold 
plane with each numbered point $i$ having height $i$ and with every semicircle
point in the closed lower halfplane having height $0$.

Suppose $\alpha$ is a long type-B arc, parametrized as a function from the interval $[0,1]$ into the orbifold plane.
The construction of a long subarc $\alpha'$ begins as follows:
Choose $t_1$ and $t_2$ with $0\le t_1<t_2\le1$ such that $h(\alpha(t_1))$ and $h(\alpha(t_2))$ are distinct elements of $\set{1,\ldots,n}$.
Choose $\ep_1>0$ such that $|h(\alpha(t_1+\ep_1))-h(\alpha(t_1))|<1$ and such that $h(\alpha(t))>0$ for all $t$ in the interval $(t_1,t_1+\ep_1)$.
%That is, the curve $\alpha$ does not contain a semicircle point between $\alpha(t_1)$ and $\alpha(t_1+\ep_1)$.
Similarly, choose $\ep_2>0$ such that $t_1+\ep_1<t_2-\ep_2$, such that  $|h(\alpha(t_2-\ep_2))-h(\alpha(t_2))|<1$, and such that $h(\alpha(t))>0$ for all $t$ in the interval $(t_2-\ep_2,t_2)$.
Define $\alpha'$ to be the \emph{curve} obtained by concatenating three curves:
First, the straight line segment from the point numbered $h(\alpha(t_1))$ to the point $\alpha(t_1+\ep_1)$;
second, the restriction of $\alpha$ to the interval $[t_1+\ep_1,t_2-\ep_2]$;
and third, the straight line segment from $\alpha(t_2-\ep_2)$ to the point numbered $h(\alpha(t_2))$.

If $\alpha'$ is a valid long arc, then it is a subarc of $\alpha$.
It fails to be a type-B arc if and only if it crosses itself, in which case, this choice of $t_1$ and $t_2$ does not produce a subarc of $\alpha$.
See \cref{fig:subarc fail orb}.
\end{remark}

\begin{figure}
    \centering
\begin{tabular}{cccc}
\includestandalone[height=6em]{figs/b/subarc.O.counterex.big} 
&& & \includestandalone[height=6em]{figs/b/subarc.O.counterex.small} \\ 
$\alpha$ &&& $\alpha'$
\end{tabular}
    \caption{A failed construction of a subarc, per \cref{subs of long,orb explicit}}
    \label{fig:subarc fail orb}
\end{figure}

The map $\phi$ translates the subarc relation on symmetric arcs/pairs exactly to the subarc relation on type-B arcs.
%(This is perhaps even more easily seen from the superarc descriptions.)
The following results are direct translations of \cref{arc forcing B,B uncontracted,B cong bij}.

\begin{theorem}\label{arc forcing B orb}
Let $j_1$ and $j_2$ be join-irreducible signed permutations.
Then $j_1$ forces $j_2$ if and only if the type-B arc corresponding to $j_1$ is a subarc of the type-B arc corresponding to $j_2$.
\end{theorem}

\begin{corollary}\label{orb uncontracted}
A set $U$ of type-B arcs corresponds to the set of \emph{un}contracted join-irreducible signed permutations of some congruence $\Theta$ on $B_n$ if and only if $U$ is closed under passing to subarcs.
\end{corollary}

\begin{corollary}\label{orb cong bij}
If $\Theta$ is a congruence on $B_n$ and $U$ is the set of type-B arcs corresponding to join-irreducible permutations \emph{not} contracted by $\Theta$, then~$\delta^\circ$ restricts to a bijection from the quotient $B_n/\Theta$ (the set of  signed permutations not contracted by $\Theta$) to the set of type-B noncrossing arc diagrams consisting only of arcs in $U$.
\end{corollary}

A type-B arc $\alpha'$ is a \newword{superarc} of a type-B arc $\alpha$ if and only if $\alpha$ is a subarc of~$\alpha'$.
We now describe how to construct superarcs. 

\medskip
\noindent
\textbf{Superarcs of an orbifold arc $\alpha$.}
These can be orbifold arcs or long arcs.
We construct an orbifold superarc by pushing $\alpha$ left or right of its top endpoint and extending it upwards to a new upper endpoint.
We construct a long superarc by first replacing $\alpha$ with a curve that goes around \orb, with both endpoints at the original endpoint of $\alpha$, and with its left piece and its right piece having exactly the same right points as $\alpha$.
We then push one or both ends off of the endpoint (one left and one staying, one right and one staying, both left, one on each side, or both right, but not creating a self-intersection), and extend upwards without creating self-intersections, to make a long superarc.

\medskip
\noindent
\textbf{Superarcs of an ordinary arc $\alpha$.}
These can be of any of the three kinds.
To construct a superarc, we push the top or bottom of $\alpha$ independently to the left or right.
We then extend upwards and/or downwards.
The downward extension can end before reaching \orb (to make an ordinary superarc), can have an endpoint at \orb (to make an orbifold superarc), or can go around \orb before ending (to make a long superarc), provided that no self-intersections are created.

\medskip
\noindent
\textbf{Superarcs of a long arc $\alpha$.}
Every superarc of $\alpha$ is long.
We construct a superarc by pushing one or both endpoints of $\alpha$ left or right and extending upwards, without creating self-intersections.

\medskip

We close this section with a translation of \cref{Emily's B arcs}, which will be useful in the sequel.
As before, we write arrows $\to$ between type-B arcs to indicate arrows between the corresponding shards.

\begin{proposition}\label{B orb arrows}  
Suppose $\alpha_1$ and $\alpha_2$ are type-B arcs.
Then $\alpha_1\to\alpha_2$ if and only if $\alpha_1$ is a subarc of $\alpha_2$ and one of the following conditions holds.  \begin{enumerate}[\qquad\rm(i)]
\item \label{nonorb nonorb}
Neither $\alpha_1$ nor $\alpha_2$ is an orbifold arc, and $\alpha_1$ and $\alpha_2$ have exactly one endpoint in common.
\item %\label{orb orb}
$\alpha_1$ and $\alpha_2$ are orbifold arcs with different upper endpoints.
\item \label{orb ord}
$\alpha_2$ is an orbifold arc and $\alpha_1$ is ordinary, with the same upper endpoint.
\item \label{long orb}
$\alpha_2$ is a long arc with endpoints $p<q$ and no numbered points between its left and right pieces, and $\alpha_1$ is an orbifold arc with upper endpoint~$p$.
\item \label{long ord}
$\alpha_2$ is a long arc with endpoints $p<q$ and no numbered points between its left and right pieces, and $\alpha_1$ is an ordinary arc with endpoints $p$ and~$q$.
\end{enumerate}
\end{proposition}
%We emphasize that condition~\eqref{nonorb nonorb} disallows arrows from a long arc to an ordinary arc with the same endpoints, but condition~\eqref{long ord} provides an exception when the long arc has no numbered points between its left and right pieces.
%Also, in~\eqref{long orb} and \eqref{long ord}, it does not matter whether $p$ or $q$ is the left endpoint.
 
% \begin{proposition}\label{B orb arrows}
% Suppose $\alpha_1$ and $\alpha_2$ are type-B arcs.
% Then $\alpha_1\to\alpha_2$ if and only if $\alpha_1$ is a subarc of $\alpha_2$ and one of the following conditions holds.  \begin{enumerate}[\qquad\rm(i)]
% \item \label{no orb 1 end}
% Neither $\alpha_1$ nor $\alpha_2$ is an orbifold arc, and $\alpha_1$ and $\alpha_2$ have exactly one endpoint in common.
% \item \label{orb 1 end}
% $\alpha_1$ and $\alpha_2$ are orbifold arcs with different upper endpoints.
% \item 
% $\alpha_2$ is a long arc with endpoints $p<q$ with no numbered points between its left and right pieces, and $\alpha_1$ is an orbifold arc with upper endpoint~$p$.
% \item 
% $\alpha_2$ is a long arc with endpoints $p<q$ and no numbered points between its left and right pieces, and $\alpha_1$ is an ordinary arc with endpoints $p$ and~$q$.
% \end{enumerate}
% \end{proposition}

\section{Congruence examples}\label{ex sec}
In this section, we demonstrate the use of \cref{orb uncontracted,orb cong bij} by characterizing some important lattice congruences of the weak order of type B.
Throughout this section, it will be convenient to say that a congruence ``contracts an arc'' as shorthand for ``contracts the join-irreducible element corresponding to that arc''.

% \section{Type B}\label{sec:B ex}

\subsection{Parabolic congruences}\label{ex:B parabolic}
We warm up with the simplest example, verifying easy known results.
Given a Coxeter system $(W,S)$, each element of $S$ is in particular a join-irreducible element of the weak order on~$W$.
A \newword{parabolic congruence} of the weak order on a Coxeter group $W$ is a congruence generated by contracting some subset~$J$ of~$S$.
(See \cite[Section~6]{congruence} or \cite[Example~10-7.1]{regions10}.)
The quotient of the weak order on $W$ modulo the parabolic congruence generated by $J$ is isomorphic to the weak order on the parabolic subgroup $W_{S\setminus J}$ \cite[Corollary~6.10]{congruence}.
To understand parabolic congruences on $W$, it is enough to understand the case where $J$ is a singleton.

When $(W,S)$ is of type $B_n$, the elements of $S$ and the associated arcs are as follows:  
The simple reflection $s_0=(-1)2\cdots n$ is the orbifold arc with upper endpoint~$1$, and for $i=1,\ldots n-1$, the simple reflection $s_i=1\cdots(i-1)(i+1)i(i+2)\cdots n$ is the ordinary arc connecting $i$ to $i+1$.

The parabolic congruence generated by contracting $s_i$ contracts all superarcs of corresponding arc.
For $s_0$, the set of these superarcs consists of all orbifold arcs and all long arcs.
Thus the elements of the quotient are the type-$B_n$ noncrossing arc diagrams on $n$ points that only have ordinary arcs.
This is in accordance with \cite[Corollary 6.10]{congruence}, which implies in this case that the quotient is isomorphic to~$S_n$.
%For $s_i$ with $i>0$, these superarcs are the ordinary arcs connecting a point $\le i$ to a point $\ge i+1$, the orbifold arcs with upper endpoint $\ge i+1$, and the long arcs such that either endpoint is $\ge i+1$.
For $s_i$ with $i>0$, these superarcs are the ordinary arcs connecting a point at or below $i$ to a point strictly above $i$, the orbifold arcs with upper endpoint strictly above $i$, and the long arcs such that one or both endpoints are strictly above~$i$.
Again, this accords with \cite[Corollary 6.10]{congruence}, which implies that the quotient is isomorphic to $B_i\times S_{n-i}$.

\subsection{Surjective homomorphisms from $B_n$ to $S_{n+1}$}\label{ex:B diagram}
A surprising result of \cite{diagram} is that surjective lattice homomorphisms exist between the weak orders on different finite Coxeter groups \emph{of the same rank}.
%A \newword{diagram homomorphism} from the weak order on $W$ to the weak order on $W'$ is a surjective lattice homomorphism that fixes $S$ pointwise.
Such homomorphisms always restrict to bijections between the corresponding sets of simple reflections.
Thus, we take  $(W,S)$ and $(W',S')$ to be finite Coxeter groups with the same rank $|S|=|S'|$ and fix a bijection $\phi:S\to S'$.
According to \cite[Theorem~1.6]{diagram}, there exists a surjective lattice homomorphism from $W$ to $W'$ restricting to $\phi$ if and only if $m(r,s)\ge m(\phi(r),\phi(s))$ for every $r,s\in S$.  

It is not hard to see that a surjective lattice homomorphism from $W$ to $W'$ restricts, for any $r,s\in S$, to a surjective lattice homomorphism from the lower interval $[1,r\join s]$ to the lower interval $[1,\phi(r)\join\phi(s)]$.
(For a more general statement, see \cite[Proposition~2.6]{diagram}.)
Less obviously a surjective homomorphism is uniquely determined by all of these restrictions \cite[Corollary~1.8]{diagram}.

We consider the case where $W=B_n$ and $W'=S_{n+1}$ (see \cite[Section~6]{diagram})
and fix $\phi$ to be the bijection sending $s_i$ to $s_{i+1}$ for $i=0,\ldots,n-1$.
In this case, $m(r,s)=m(\phi(r),\phi(s))$ for almost every pair $r,s\in S$, and for those pairs the only possibility is that the restriction of the homomorphism to $[1,r\join s]$ is an isomorphism.
The one exception is that $m(s_0,s_1)=4$ and $m(\phi(s_0),\phi(s_1))=m(s_1,s_2)=3$.
The interval $[1,s_0\join s_1]$ is an octagon and the interval $[1,s_1\join s_2]$ is a hexagon.
There are precisely~$4$ surjective lattice homomorphisms from the octagon to the hexagon.  
Specifically, the congruence associated to a surjective homomorphism from the octagon to the hexagon contracts exactly one of the join-irreducible elements $s_0s_1$ or $s_0s_1s_0$ and exactly one of the join-irreducible elements $s_1s_0$ or $s_1s_0s_1$. 
As a consequence, for each $n$, there are precisely $4$ surjective lattice homomorphisms from $B_n$ to $S_{n+1}$ that restrict to $\phi$, each contracting exactly one of $s_0s_1$ or $s_0s_1s_0$ and exactly one of $s_1s_0$ or $s_1s_0s_1$.  

For each surjective homomorphism from $B_n$ to $S_{n+1}$, there is a natural injection from $S_{n+1}$ into $B_n$.
In \cite{TharpThesis} (see also \cite{ReadingTharp}),
these injections for each of the $4$ homomorphisms are characterized in terms of noncrossing arc diagrams of types A and B, and $3$ of the injections are shown to embed the shard intersection order on $S_{n+1}$ as a sublattice of the shard intersection order on $B_n$.
Here, we show how the subarc/superarc tools developed in this paper recover results about these $4$ homomorphisms much more easily than in~\cite{diagram}.

Specifically, we construct, in terms of type-B noncrossing arc diagrams, four congruences, each contracting exactly one of $s_0s_1$ or $s_0s_1s_0$ and exactly one of $s_1s_0$ or $s_1s_0s_1$.  
In three cases, contracting these two join-irreducible elements generates a congruence associated to surjective homomorphisms from $B_n$ to $S_{n+1}$.
The exception is the congruence generated by contracting $s_0s_1s_0$ and $s_1s_0$.  
In this case, to get the congruence associated to a homomorphism from $B_n$ to $S_{n+1}$, we must contract additional join-irreducible elements.
Everything we need to know about these congruences follows from some easy facts about arcs and superarcs:  %, recorded in the following table.
\begin{itemize}
\item
$s_0s_1=2(-1)3\cdots n$ corresponds to the long arc with left endpoint~$1$ and right endpoint $2$, with right piece passing right of~$1$.
Its superarcs are the long arcs whose right piece passes right of $1$.
\item
$s_0s_1s_0=(-2)(-1)3\cdots n$ corresponds to the orbifold arc with upper endpoint $2$, passing right of $1$.
Its superarcs are the orbifold arcs that pass right of $1$ and the long arcs both of whose pieces pass right of $1$.
\item 
$s_1s_0=(-2)13\cdots n$ corresponds to the orbifold arc with upper endpoint $2$, passing left of $1$.
Its superarcs are the orbifold arcs that pass left of $1$ and the long arcs both of whose pieces pass left of $1$.
\item 
$s_1s_0s_1=1(-2)3\cdots n$ corresponds to the long arc with right endpoint $1$ and left endpoint $2$, with left piece passing left of~$1$.
Its superarcs are the long arcs whose left piece passes left of $1$.
\item 
$s_1s_0s_1s_2=13(-2)4\cdots n$ corresponds to the long arc with left endpoint $2$ and right endpoint $3$, with left piece passing left of~$1$ and right piece passing right of $1$ and $2$.
Its superarcs are the long arcs whose left piece passes left of~$1$ and whose right piece passes right of $1$ and $2$.
\item 
$s_2s_1s_0s_1s_2=12(-3)4\cdots n$ corresponds to the long arc with right endpoint $2$ and left endpoint $3$, with right piece passing right of~$1$ and left piece passing left of $1$ and $2$.
Its superarcs are the long arcs whose right piece passes right of~$1$ and whose left piece passes left of $1$ and $2$.
\end{itemize}

These easy facts, along with \cref{orb uncontracted}, lead almost immediately to the proofs of the following three propositions, which are \cite[Proposition~6.3]{diagram}, \cite[Proposition~6.6]{diagram}, and \cite[Proposition~6.9]{diagram}, translated from the language of shards to the language of type-B arcs.
The missing fourth proposition (describing the congruence generated by contracting $s_0s_1$ and $s_1s_0$) is symmetric to the third by swapping left and right, and this symmetry is also exploited in~\cite{diagram}.

\begin{proposition}\label{simion finest}
The congruence generated by contracting $s_0s_1$ and $s_1s_0s_1$ contracts the long arcs and no other arcs.
\end{proposition}
\begin{proof}
This congruence contracts precisely the arcs that are long, with right piece passing right of $1$ or left piece passing left of $1$, or in other words, all long arcs.
\end{proof}

\begin{proposition}\label{nonhom finest}
The congruence generated by contracting $s_0s_1s_0$, $s_1s_0$, $s_1s_0s_1s_2$, and $s_2s_1s_0s_1s_2$ contracts an arc if and only if it is not ordinary and only has endpoints $\ge2$.
\end{proposition}
\begin{proof}
This congruence contracts no ordinary arcs. 
It contracts an orbifold arc if and only if it passes left or right of $1$, or in other words, if and only if its upper endpoint is $\ge2$.
It contracts a long arc if and only if both of its pieces pass right of $1$ or both of its pieces pass left of $1$ or its left piece passes left of $1$ and its right piece passes right of $1$.
In other words, it contracts a long arc if and only if its endpoints are both $\ge2$.
\end{proof}

\begin{proposition}\label{delta finest}
The congruence generated by contracting  $s_0s_1s_0$ and $s_1s_0s_1$ contracts an arc if and only if it is an orbifold arc passing right of ~$1$ or a long arc whose left endpoint is not~$1$.
\end{proposition}
\begin{proof}
This congruence also contracts no ordinary arcs. 
It contracts an orbifold arc if and only if it passes right of $1$.
It contracts a long arc if and only if both of its pieces pass right of $1$ or its left piece passes left of $1$.
In other words, it contracts a long arc if and only if its left endpoint is not $1$.
\end{proof}

\subsection{Cambrian congruences}\label{ex:B cambrian}
Let $(W,S)$ be a finite Coxeter system.
%Let $(W,S)$ be an irreducible finite Coxeter system (a Coxeter system whose Coxeter diagram is connected).  
%(It is well known that the Coxeter diagram of an irreducible finite Coxeter system is a tree.)
A Cambrian congruence on the weak order on $W$ is defined by choosing a set of join-irreducible elements to contract in the following way:
For each edge in its Coxeter diagram, connecting distinct simple reflections $r,s\in S$ and labeled $m(r,s)$, choose either to contract every join-irreducible element $rs,rsr,rsrsr,\ldots$ up to length $m(r,s)-1$ or to contract every join-irreducible element $sr,srs,srsr,\ldots$ up to length $m(r,s)-1$.
(Since most edges of a Coxeter diagram of finite type are labeled $3$, for most edges the choice is between one join-irreducible element and another.)
The \newword{Cambrian congruence} is the congruence generated by contracting the chosen join-irreducible elements.
(Often, these choices are described as the choice of a Coxeter element of $W$ or as the choice of an orientation of the Coxeter diagram, but we do not need to be so specific here.)  
The quotient of the weak order modulo a Cambrian congruence is called a \newword{Cambrian lattice}.

To give context to Cambrian lattices of type B, we review the description of Cambrian lattices of type~A in terms of noncrossing arc diagrams, following \cite[Example~4.9]{arcs}, and add some straightforward observations.
Each numbered point from $2$ to $n-1$ is designated as either a left point or a right point.  
Designating $i$ as a right point corresponds to contracting the join-irreducible element $s_{i-1}s_i$, while designating $i$ as a left point corresponds to contracting $s_is_{i-1}$.
The superarcs of these chosen contracted join-irreducible elements are precisely the arcs that pass left of a left point or right of a right point.
\cref{A uncontracted,A cong bij} imply that $\delta$ restricts to a bijection from the elements of the Cambrian lattice to the noncrossing arc diagrams containing arcs that only pass right of left points and only pass left of right points.

In the cases where there are no right points or no left points, the corresponding Cambrian lattices are \newword{type-B Tamari lattices} in the sense of \cite{cambrian,Thomas}. 
The two type-B Tamari lattices are dual to each other.

A convenient way to draw noncrossing arc diagrams for elements of the Cambrian lattice is to draw a horizontal ``barrier'' ray going right from each right point and going left from each left point.
Elements of the Cambrian lattice are then noncrossing arc diagrams such that every arc is disjoint from the barriers, except at endpoints.
For example, the left picture of \cref{fig:ncp A} shows the noncrossing arc diagram for the permutation $497862153$, which is an element of the Cambrian lattice generated by contracting $s_{i-1}s_i$ for $i\in\set{2,3,5,8}$ and $s_is_{i-1}$ for $i\in\set{4,6,7}$.

\begin{figure}
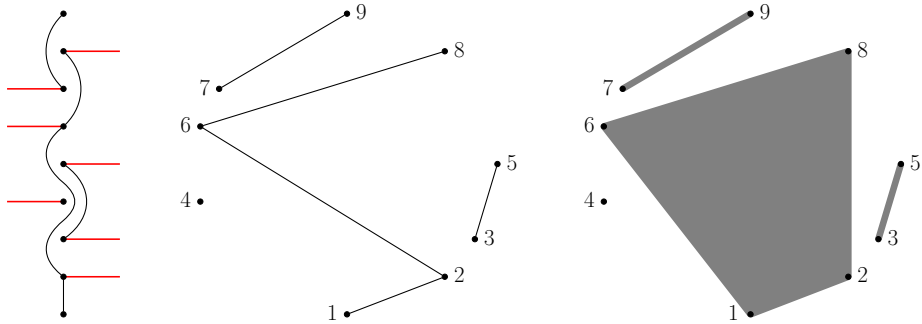

    \centering
\begin{tabular}{ccccc}
\includestandalone[scale=0.5]{figs/a/ncp.straight} 
&\qquad&
\includestandalone[scale=0.5]{figs/a/ncp.circle} 
&\qquad&
\includestandalone[scale=0.5]{figs/a/ncp.blocks} 
\end{tabular}
    \caption{Cambrian lattice and noncrossing partitions, type A}
    \label{fig:ncp A}
\end{figure}

Another way to draw noncrossing arc diagrams for elements of the Cambrian lattice highlights the bijection between the Cambrian lattice and the set of noncrossing partitions of an appropriate cycle.
The numbers $1$ and $n$ are placed at the bottom and top of a circle respectively, and the numbers $2,\ldots,n-1$ are placed on the circle, increasing in height as the numbers increase in value, with left points on the left of the circle and right points on the right of the circle.
The noncrossing arc diagrams for elements of the Cambrian lattice can then be drawn with ``arcs'' that never leave the circle (and may as well be straight line segments).
For example, the right picture of \cref{fig:ncp A} shows the same arc diagram as the left picture, drawn with points on a circle.

The numbered points define a cycle, starting with $1$, passing through the left points in increasing order, then $n$, then the right points in decreasing order.
The bijection between elements of the Cambrian lattice and noncrossing partitions of the cycle is thus clear:
Each connected component of the noncrossing arc diagram contains some set of numbered points, and these sets of numbered points are the blocks of the noncrossing partition. 
Inversely, given a noncrossing partition $\P$ of the cycle, the noncrossing arc diagram corresponding to $\P$ consists of the following arcs:
For each block in $\P$ with elements $b_1<b_2<\cdots<b_k$, there is an arc connecting $b_{i+1}$ to $b_i$ for each $i=1,\ldots,k-1$, with no arc passing left of a left point or right of a right point.

There is another interesting characterization of the Cambrian congruences of type A due to Santocanale and Wehrung~\cite{SanWeh}.
To state this characterization, we recall some facts about the congruence lattice.
The congruences of a fixed finite lattice~$L$ constitute a distributive lattice $\Con(L)$ under refinement order (and in fact a sublattice of the partition lattice).
The join-irreducible congruences are the congruences generated by contracting a single join-irreducible element of the lattice.
The meet-irreducible congruences are also indexed by join-irreducible elements of the lattice:  
A meet-irreducible congruence is the largest (i.e.\ coarsest) congruence not contracting some join-irreducible element.

Although Cambrian congruences are defined as the join of join-irreducible congruences (the congruences generated by each contracted element), remarkably, when~$W$ is of type A, the Cambrian congruences are themselves meet-irreducible congruences.
Specifically, \cite[Corollary 6.10]{SanWeh} says that the Cambrian congruences of the weak order in type A are precisely the \emph{minimal} meet-irreducible congruences of the weak order.
This statement also has a simple proof in terms of noncrossing arc diagrams:
%A minimal meet-irreducible congruence is a meet-irreducible congruence associated to an arc that has no proper superarcs. 
A meet-irreducible congruence is minimal if and only if it is associated to an arc that has no proper superarcs.
After designating the points $2,\ldots,n-1$ as left and right points, there is a unique longest arc $\alpha$ that goes right of every left point and left of every right point.
This arc $\alpha$ has no proper superarcs, and the arcs not contracted by the Cambrian congruence are precisely the subarcs of~$\alpha$.
In other words, the Cambrian congruence is the meet-irreducible congruence associated to~$\alpha$.

In this section, we use \cref{orb uncontracted,orb cong bij} to characterize Cambrian congruences of type B combinatorially and lattice-theoretically.
The lattice-theoretic characterization is related to \cite[Corollary 6.10]{SanWeh} and characterizes Cambrian congruences in type B as meets of meet-irreducible congruences.
%In \cref{B min mi}, we generalize the results of \cite{SanWeh} in a different direction, by characterizing the minimal meet-irreducible congruences of the weak order in type~B as joins of join-irreducible congruences.
%\marginN{\label{mention here}  If we indeed leave out \cref{B min mi} after the first arXiv version (see Note~\ref{leave out}), this would be where we could mention that it's in that arXiv version. ``One can also generalize the results of \cite{SanWeh} in a different direction, by characterizing the minimal meet-irreducible congruences of the weak order in type~B as joins of join-irreducible congruences.  This characterization can be found in the first arXiv version of this paper.''}
%(We do not treat the problem of characterizing the minimal meet-irreducible congruences of the weak order in type~B, although that problem is within reach of the subarc techniques established here.)

As part of determining a Cambrian congruence of the weak order on~$B_n$, we choose to contract either $s_0s_1=2(-1)3\cdots n$ and $s_0s_1s_0=(-2)(-1)3\cdots n$, or $s_1s_0=(-2)13\cdots n$ and $s_1s_0s_1=1(-2)3\cdots n$.
In in both cases, the corresponding arcs are an orbifold arc with upper endpoint $2$ and a long arc with endpoints $1$ and~$2$.
In one case, both arcs pass right of $1$ and an arc is a superarc of one or both of them if and only if some part of it passes right of~$1$.
In the other case both arcs pass left of $1$, and the superarcs are those that pass left of~$1$.

The remaining choices that determine a Cambrian congruence are, for each ${i=2,\ldots n-1}$, to contract $s_{i-1}s_i=1\cdots(i-2)\,i\,(i+1)(i-1)(i+2)\cdots n$ or to contract $s_is_{i-1}=1\cdots(i-2)(i+1)(i-1)\,i\,(i+2)\cdots n$.
The arc corresponding to $s_{i-1}s_i$ connects $i+1$ to $i-1$ and passes right of~$i$.
Its superarcs are precisely the arcs that have some part passing right of $i$.
The arc corresponding to $s_is_{i-1}$ connects $i+1$ to $i-1$ and passes left of~$i$, and its superarcs are the arcs that pass left of $i$.

We see that choosing a Cambrian congruence on the weak order on $B_n$ corresponds to designating each point from $1$ to $n-1$ as a left point or right point.
By \cref{orb uncontracted}, the Cambrian congruence contracts all arcs that pass right of a right point and/or left of a left point.
Thus \cref{orb cong bij} implies the following theorem.

\begin{theorem}\label{B Camb thm}
Given a Cambrian congruence and the corresponding designation of $1,\ldots,n-1$ as left and right points, the map~$\delta^\circ$ on signed permutations restricts to a bijection from the associated Cambrian lattice to the set of type-B noncrossing arc diagrams consisting of arcs that never pass left of a left point and never pass right of a right point.
\end{theorem}

For example, the left pictures of \cref{fig:ncp B} show the type-B noncrossing arc diagram for two elements of the Cambrian lattice for a particular choice of left and right points.
\begin{figure}
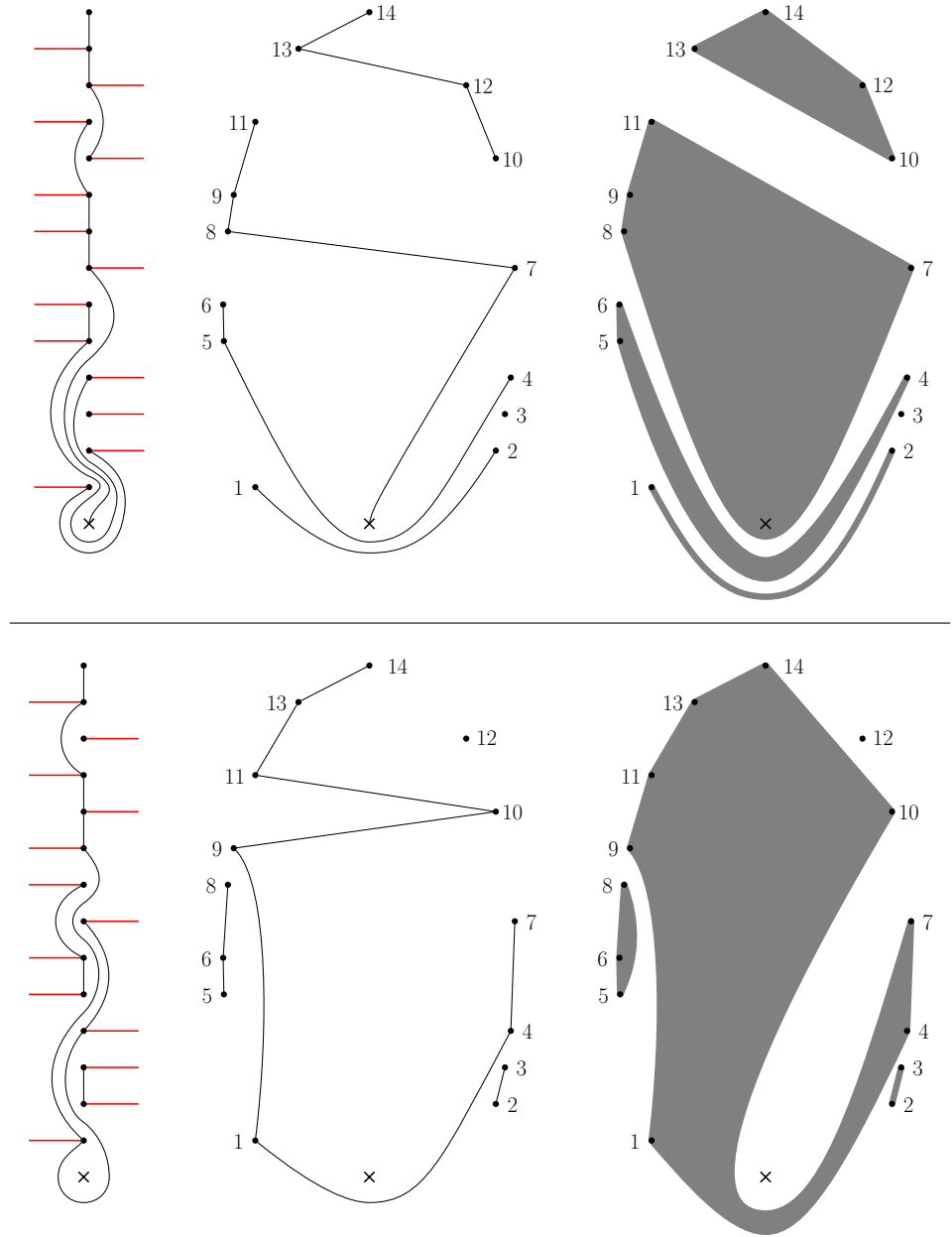

    \centering
\begin{tabular}{ccccc}
\includestandalone[scale=0.49]{figs/b/ncp.straight} 
&\qquad&
\includestandalone[scale=0.49]{figs/b/ncp.circle} 
&\qquad&
\includestandalone[scale=0.49]{figs/b/ncp.blocks} 
\\[5pt]\hline\\
\includestandalone[scale=0.49]{figs/b/ncp2.straight} 
&\qquad&
\includestandalone[scale=0.49]{figs/b/ncp2.circle} 
&\qquad&
\includestandalone[scale=0.49]{figs/b/ncp2.blocks} 
\end{tabular}
    \caption{Cambrian lattice and noncrossing partitions, type B}
    \label{fig:ncp B}
\end{figure}

If one works in the centrally symmetric model (and writes down a centrally symmetric version of \cref{B Camb thm} in the obvious way), one can connect elements of the Cambrian lattice with centrally symmetric noncrossing partitions by imposing a central symmetry conditiont throughout the procedure described above and exemplified in \cref{fig:ncp A}.
Here, we will stay in the orbifold model and use \cref{B Camb thm} to describe the bijection between elements of the Cambrian lattice and noncrossing partitions of a disk with one orbifold point.  
(See the discussion in \cite[Section~4.2]{surfnc} or make the obvious modification of the discussion in \cite[Section~4.2]{affncA} of noncrossing partitions of a disk with two orbifold points.)

Consider a closed disk $D$ with one distinguished point \orb in its interior and numbered points $1,\ldots,n$ on its boundary.
Given a subset of the numbered points, an \newword{embedding} of the subset is the choice of a closed disk $B\subset D$ containing the numbered points in that block but otherwise disjoint from the boundary of $D$.
The disk $B$ may be disjoint from \orb or may contain \orb in its interior.
(The disk $B$ may not contain \orb on its boundary.)  
A \newword{noncrossing partition of the disk with one orbifold point} is a set partition of $\set{1,\ldots,n}$ together with a choice, for each block of the partition, of an embedding, such that the embeddings of any two blocks are disjoint.
These noncrossing partitions are considered up to isotopy, fixing the boundary of $D$ and the point~\orb.
The natural partial order on these noncrossing partitions is the usual type-B noncrossing partition lattice.

The designation of each point in $\set{1,\ldots,n-1}$ as a left point or a right point determines a placement of numbered points on the boundary of $D$ as follows.
The number $n$ is placed at the top of $D$, and the numbers $1,\ldots,n-1$ are placed on the boundary of $D$, increasing in height as the numbers increase in value, with left points on the left and right points on the right, and \orb is placed lower than 1 in the interior of $D$.

The bijection between elements of the Cambrian lattice and noncrossing partitions of the one-orbifold disk is now apparent, and is illustrated in \cref{fig:ncp B}.
Each block of the noncrossing arc diagram contains some set of numbered points, and these sets of numbered points are the blocks of the noncrossing partition.
The embedding of each block is obtained from the noncrossing arc diagram by moving left points to the left, onto the left boundary of $D$ and right points onto the right boundary, moving the arcs correspondingly, and choosing a disc that contains the arcs in the block (or the unique point in the block if it has no arcs).
Inversely, given a noncrossing partition of the one-orbifold disk, we can read off the corresponding type-B noncrossing arc diagram.
Specifically, we read from the noncrossing partition certain pairs of points that should be connected by arcs in the type-B noncrossing arc diagram.
The arcs that connect the pairs must, of course, satisfy the requirement of \cref{B Camb thm}, in that they never pass left of a left point or right of a right point.
If there is a block containing~\orb, then we connect the lowest numbered point in the block to \orb by an orbifold arc and connect each other numbered point in the block to the next lower numbered point in the block by an ordinary arc. 
For each block that passes below \orb, we first treat the block as if it had two pieces, separated at the part of the block that goes below~\orb.
We connect each numbered point within each piece (except the lowest numbered point in the piece) to the next lower numbered point in that piece by an ordinary arc.
We then connect the lowest point in one piece to the lowest point in the other piece by a long arc.
For each block that passes above~\orb, we connect each numbered point in the block, except the lowest, to the next lower numbered point in the block, by an ordinary arc.

\cref{B Camb thm} also lets us recover (much more easily) the characterization of the type-B Cambrian lattices given in \cite{cambrian}:
The Cambrian lattice consists of the signed permutations whose long one-line notation has no patterns $231$ such that the ``$2$'' is a right point or the negative of a left point.
More formally, we have the following theorem, which is \cite[Theorem~7.5]{cambrian}, restated in the language of this section.

\begin{theorem}\label{B Camb pat thm}
Given a designation of $1,\ldots,n-1$ as left and right points, the associated Cambrian lattice is the subposet of the weak order induced by signed permutations whose long one-line notation has no subsequence $bca$ with $a<b<c$ such that $b$ is a right point or the negative of a left point.
\end{theorem}
By the symmetry of the long one-line notation, the criterion in \cref{B Camb pat thm} is equivalent to requiring that the long one-line notation has no pattern $312$ such that the ``$2$'' is a left point or the negative of a right point.

To write the Cambrian congruence as a meet of meet-irreducible congruences, we need to find the arcs that are not contracted by the Cambrian congruence but all of their proper superarcs are contracted. 
In every case, one such arc is the unique orbifold arc $\beta_o$ with upper endpoint $n$ that passes left of every right point and right of every left point.
Besides this orbifold arc, there are one or two additional uncontracted arcs that have no uncontracted proper superarcs.
Each is a long arc with nothing between its left and right pieces, passing left of every right point and right of every left point.
One arc, $\beta_r$ has left endpoint at $n$ and right endpoint at the highest right point and the other, $\beta_l$, has right endpoint at $n$ and left endpoint at the highest left point.
If there are no right points or no left points (the case where the Cambrian lattice is a type-B Tamari lattice), then only one of the arcs exists $\beta_r$ or $\beta_l$ exists.  
\cref{fig:meetand B} shows $\beta_o$, $\beta_r$, and $\beta_l$ for a particular choice of right and left points, in the case $n=6$.
\begin{figure}
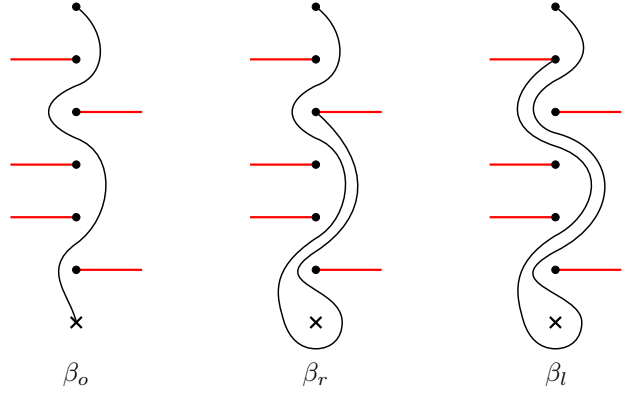

    \centering
\begin{tabular}{ccccccccc}
\includestandalone[scale=0.7]{figs/b/meetand.o} 
&&&&
\includestandalone[scale=0.7]{figs/b/meetand.r} 
&&&&
\includestandalone[scale=0.7]{figs/b/meetand.l} \\
$\beta_o$&&&&$\beta_r$&&&&$\beta_l$
\end{tabular}
    \caption{A meet representation of the Cambrian congruence}
    \label{fig:meetand B}
\end{figure}

Let $\Theta_o$ be the largest (i.e.\ coarsest) congruence that does not contract $\beta_o$.
Define $\Theta_r$ and/or $\Theta_l$ similarly (if $\beta_r$ and/or $\beta_l$ exist).
We have proven the following theorem.  

\begin{theorem}\label{Cambrian meet}
Let $\Theta$ be the Cambrian congruence on $B_n$ associated to a designation of $1,\ldots,n-1$ as left and right points.
Then $\Theta$ is the meet, in the lattice of congruences, of the meet-irreducible congruences $\Theta_o$, $\Theta_r$, and $\Theta_l$ (but leaving out~$\Theta_r$ if there are no right points or leaving out $\Theta_l$ if there are no left points).
\end{theorem}

\subsection{Bipartite biCambrian congruences}\label{ex:B bip bicambrian}  
Recall from \cref{ex:B cambrian} that a Cambrian congruence on the weak order on $W$ is defined by a binary choice for each edge in the Coxeter diagram.
Two Cambrian congruences are \newword{opposite} if they arise from making the opposite choice for each edge.
A \newword{biCambrian congruence} is the meet, in the congruence lattice of the weak order, of two opposite Cambrian congruences.
(The meet of two congruences contracts an element if and only if both congruences contract that element.)
More details on biCambrian congruences can be found in \cite[Section~2.3]{bicat}.

A Cambrian congruence is \newword{bipartite} if, whenever $r$ and $s$ form an edge and~$s$ and~$t$ also form an edge, the congruence either contracts the join-irreducible elements $rs,rsr,rsrsr,\ldots$ and $ts,tst,tstst,\ldots$ or contracts the join-irreducible elements $sr,srs,srsr,\ldots$ and $st,sts,stst,\ldots$.
Assuming $W$ is \emph{irreducible}, there are precisely two biparitite Cambrian congruences, and they are opposite each other.
The \newword{bipartite biCambrian congruence} is the biCambrian congruence that is the meet of these two bipartite Cambrian congruences.

We briefly review the combinatorics of bipartite biCambrian congruences in type~A.
(More details are in \cite[Section~3.3]{bicat}.)
One bipartite Cambrian congruence corresponds to setting all even numbers in $2,\ldots,n-1$ to be left points and all odd numbers to be right points.
The other is the same, but with left and right switched.
Thus the bipartite biCambrian congruence contracts an arc if and only if it goes right of both an even number and an odd number or goes left of both an even number and an odd number.
The uncontracted arcs are the \newword{alternating arcs}, which don't pass two consecutive points on the same side (and thus alternate left-right or right-left).
Every arc with endpoints differing by less than 3 is alternating, because it doesn't pass two consecutive points at all.
Thus the smallest non-alternating arcs, in the sense of subarcs, are the arcs with endpoints $i$ and $i+3$, passing to the same side of $i+1$ and $i+2$, for $i=1,\ldots,n-3$.
The bipartite biCambrian congruence is thus generated by contracting the join-irreducible elements $s_is_{i+1}s_{i+2}$ and $s_{i+2}s_{i+1}s_i$ corresponding to these arcs.
The permutations not contracted by the bipartite biCambrian congruence are described by a bivincular pattern avoidance condition in the sense of~\cite[Section~2]{BCDK}.
(See \cite[Proposition~3.7]{bicat}.)

We now turn to type~B.
The bipartite biCambrian congruence on signed permutations is described in \cite[Section~3.5]{bicat} in terms of symmetric alternating arc diagrams.
Here, we give an account in the orbifold model and use the type-B subarc relation to add a new description of the generators of the congruence.

Similarly to type A, one bipartite Cambrian congruence of type B corresponds to designating all even numbers in $1,\ldots,n-1$ to be left points and all odd numbers to be right points.
The other is the same with right and left reversed.  
Thus the uncontracted arcs of the bipartite biCambrian congruence of type B are the \newword{alternating type-B arcs}, meaning:
\begin{itemize}
\item ordinary arcs that don't pass consecutive points on the same side,
\item orbifold arcs that don't pass consecutive points on the same side, and
\item long arcs with no numbered points between them, with left piece and right piece each not passing pass consecutive points on the same side.
\end{itemize}

The following theorem is easily verified by checking that the type-B arcs listed are the minimal non-alternating type-B arcs, in the sense of subarcs.  

\begin{theorem}\label{B bicat gen bip}
The type-B bipartite biCambrian congruence is generated by contracting the following join-irreducible elements.
\begin{itemize}
\item for each $i=1,\ldots,n-3$, the join-irreducible elements $s_is_{i+1}s_{i+2}$ and $s_{i+2}s_{i+1}s_i$ corresponding to ordinary arcs with endpoints $i$ and $i+3$ passing to the same side of $i+1$ and $i+2$,
\item the join-irreducible elements $s_0s_1s_0s_2s_1s_0$ 
% -123
% 2-13
% -2-13
% -23-1
% 3-2-1
% -3-2-1
and $s_2s_1s_0$ corresponding to orbifold arcs with upper endpoint $3$ passing to the same side of $1$ and $2$, and 
\item the join-irreducible elements 
$s_0s_1s_2$ 
% -123
% 2-13
% 23-1
and $s_2s_1s_0s_1$ 
% 132
% 312
% -312
% 1-32
corresponding to long arcs with endpoints at $1$ and $3$ passing to the same side of $1$ and $2$.
\end{itemize}
\end{theorem}

\subsection{Linear biCambrian congruences}\label{ex:B lin bicambrian}  
When $W$ is of type A or B (or F or H, which are not considered in this paper), there are two \newword{linear Cambrian congruences}.  
(The corresponding Cambrian lattices are the Tamari lattices and type-B Tamari lattices.
See~\cite{cambrian,Thomas}.)
One of them contracts all join-irreducible elements of the form $s_is_{i+1},s_is_{i+1}s_i,\ldots$, and the other 
contracts all join-irreducible elements of the form $s_{i+1}s_i,s_{i+1}s_is_{i+1},\ldots$.
The \newword{linear biCambrian congruence} is the meet of these two opposite linear Cambrian congruences.

In type A, the linear biCambrian congruence contracts an arc if and only if it goes left of some point and right of some other point.  
It is generated by contracting, for each $i=1,\ldots,n-3$, the join-irreducible elements $s_is_{i+2}s_{i+1}$ and $s_{i+1}s_is_{i+2}s_{i+1}$, associated to arcs with endpoints $i$ and $i+3$, passing to opposite sides of $i+1$ and $i+2$.
The permutations not contracted by the linear biCambrian congruence called the \newword{twisted Baxter permutations}, are defined by a vincular pattern  and are counted by the Baxter number.
(See \cite[Section~10]{con_app} and \cite[Section~8]{rectangle}.)

In type B, the linear biCambrian congruences also contracts an arc if and only if it goes left of some point and right of another.  

\begin{theorem}\label{B bicat gen lin}
The type-B linear biCambrian congruence is generated by contracting the following join-irreducible elements.
\begin{itemize}
\item for each $i=1,\ldots,n-3$, the join-irreducible elements $s_is_{i+2}s_{i+1}$ and $s_{i+1}s_is_{i+2}s_{i+1}$ corresponding to ordinary arcs with endpoints $i$ and $i+3$ passing to opposite sides of $i+1$ and $i+2$,
\item the join-irreducible elements $s_0s_2s_1s_0$ 
% -123
% -132
% 3-12
% -3-12
and $s_1s_2s_0s_1s_0$ 
% -3-21
corresponding to orbifold arcs with upper endpoint $3$ passing to opposite sides of $1$ and $2$, and 
\item the join-irreducible elements 
$s_0s_2s_1$ 
% -123
% -132
% 3-12
and $s_1s_0s_1s_2s_1s_0s_1$ 
%s1: 213
%s0: -213
%s1: 1-23
%s2s1: 31-2
%s0: -31-2
%s1: 1-3-2
corresponding to long arcs with endpoints at $1$ and $3$ passing to opposite sides of $1$ and $2$.
\end{itemize}
\end{theorem}

The theorem is easily verified by checking that the arcs listed are minimal, in the sense of subarcs, among arcs that pass left of some point and right of some other point.

\subsection{Symmetric congruences of the weak order on $S_{2n}$}  
The weak order on signed permutations is a sublattice of the weak order on all permutations of $\set{\pm1,\ldots,\pm n}$. 
(In light of \cref{w0 arc} and as discussed in \cref{B sym sec}, this sublattice consists of all permutations whose corresponding noncrossing arc diagrams are centrally symmetric.)

Given a lattice and a sublattice, any congruence on the lattice restricts to a congruence on the sublattice.
However, there may exist congruences on the sublattice that are not the restriction of congruences on the larger lattice.
(See \cite[Remark~119]{shardpoly}.)  
We write $\Con_A(B_n)$ for the set of congruences on the weak order on $B_n$ that are the restriction of congruences on $S_{2n}$.
The purpose of this section is to characterize $\Con_A(B_n)$.

Recall that conjugation by $w_0$ is an involutive automorphism of the weak order on $S_{2n}$, and when $S_{2n}$ is represented as permutations of $\set{\pm1,\ldots,\pm n}$, this map is 
\[\pi_{-n}\cdots\pi_{-2}\pi_{-1}\pi_1\pi_2\cdots\pi_n\mapsto (-\pi_n)\cdots(-\pi_2)(-\pi_1)(-\pi_{-1})(-\pi_{-2})\cdots(-\pi_{-n}).\] 
A \newword{symmetric congruence} of $S_{2n}$ is a congruence such that $\pi\equiv\tau$ if and only if $w_0\pi w_0\equiv w_0\tau w_0$ for all $\pi$ and $\tau$.

\begin{definition}\label{loose sub def}
A type-B arc $\alpha'$ is a \newword{loose subarc} of a type-B arc $\alpha$ if and only if one of the following conditions holds:
\begin{itemize}
\item 
$\alpha'$ is a subarc of $\alpha$ in the sense of \cref{subs of ord/orb} or \cref{subs of long};
\item 
$\alpha$ is an orbifold arc with endpoint $q$ and $\alpha'$ is a long arc with left endpoint $p'$ and right endpoint $q'$ with such that $L(\alpha')=L(\alpha)\cap(0,p')$ and $R(\alpha')=R(\alpha)\cap(0,q')$.
(In particular, there are no numbered points between the two sides of $\alpha'$.)
\item
$\alpha$ is a long arc with left endpoint $p$ and right endpoint $q$ and $\alpha'$ is a long arc with left endpoint $p'$ and right endpoint $q'$ with $p'\le q$ and $q'\le p$, $L(\alpha')=(0,p')\setminus R(\alpha)$, $R(\alpha')=(0,q')\setminus L(\alpha)$, and there are no numbered points that are between the two pieces of $\alpha'$.
\end{itemize}
\end{definition}

We emphasize the comparison between \cref{loose sub def} and \cref{subs of ord/orb,subs of long}.
First, \cref{loose sub def} allows long loose subarcs of orbifold arcs, while \cref{subs of ord/orb} disallows long subarcs of orbifold arcs.
Second, \cref{loose sub def} allows, loosely speaking, the \emph{left} piece of $\alpha'$ to be obtained by cutting the \emph{right} piece of $\alpha$ and the \emph{right} piece of  $\alpha'$ to be obtained by cutting the \emph{left} piece of $\alpha$ as long as there are no numbered points between the pieces of $\alpha'$.
(The definition of loose subarcs is equivalent to removing the additional requirement in \cref{subs of pair} for subarc pairs of a symmetric pair as long as there are no numbered points between the two arcs in the subarc pair.)
\cref{fig:loose} shows some pairs $\alpha'$ and $\alpha$ of type-B arcs such that $\alpha'$ is a loose subarc of $\alpha$ but not a subarc of $\alpha$.

\begin{figure}
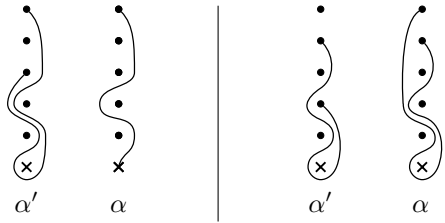

\centering
\begin{tabular}{ccc|ccc}
\begin{tabular}{cccc}
\includestandalone[scale=0.6]{figs/b/loose.orb.small}
&&
\includestandalone[scale=0.6]{figs/b/loose.orb.big}
\\
$\alpha'$&&$\alpha$
\end{tabular}
&&&&&
\begin{tabular}{cccc}
\includestandalone[scale=0.6]{figs/b/loose.long.small}
&&
\includestandalone[scale=0.6]{figs/b/loose.long.big}
\\
$\alpha'$&&$\alpha$
\end{tabular}
\end{tabular}
\caption{Loose subarcs that are not subarcs}
\label{fig:loose}
\end{figure}

The following \lcnamecref{sym cong} characterizes $\Con_A(B_n)$. % the congruences on $B_n$ that are restrictions of congruences on $S_{2n}$.

\begin{theorem}\label{sym cong}
If $\Theta$ is a lattice congruence on the weak order on $B_n$, the following are equivalent.
\begin{enumerate}[\qquad\rm(i)]
\item \label{res cong}
$\Theta$ is the restriction of a congruence on $S_{2n}$.
\item \label{res sym cong}
$\Theta$ is the restriction of a symmetric congruence on $S_{2n}$. That is, $\Theta$ is in~$\Con_A(B_n)$.
\item \label{loose sub}
The set of type-B arcs corresponding to join-irreducible signed permutations \emph{not} contracted by $\Theta$ is closed under passing to loose subarcs.
\end{enumerate}
\end{theorem}
\begin{proof}
The equivalence of \eqref{res cong} and \eqref{res sym cong} is a special case of a standard lattice argument.
Suppose $\Theta$ is the restriction of some congruence $\Theta'$ on $S_{2n}$ and let $\Theta''$ be the congruence on $S_{2n}$ 
with $\pi\equiv\tau$ modulo $\Theta''$ if and only if $w_0\pi w_0\equiv w_0\tau w_0$ modulo~$\Theta'$.
Then $\Theta''$ also restricts to $\Theta$.
The meet in the lattice of congruences is coarsest common refinement, so $\Theta'\meet\Theta''$ also restricts to $\Theta$ and is manifestly symmetric.

Suppose $\Theta$ is a congruence on $B_n$ and let $U$ be the set of type-B arcs corresponding to join-irreducibles not contacted by $\Theta$.
Since $\Theta$ is a congruence, $U$ is closed under passing to subarcs in the usual type-B sense.  
We need to show that there exists a congruence $\Theta'$ of $S_{2n}$ that restricts to $\Theta$ if and only if $U$ is also closed under passing to loose subarcs that are not subarcs, as described in the second and third bullet points of \cref{loose sub def}.  

Given a type-B arc $\alpha$, write $\tilde\alpha$ for the symmetric arc/pair associated to $\alpha$.
(In the language of \cref{B orb sec}, $\alpha=\phi(\tilde\alpha)$.)
In both directions of the proof, the essential point is the following observation, which is easily verified, perhaps with the help of \cref{fig:loose}:
Given type-B arcs $\alpha'$ and $\alpha$, the case where $\alpha'$ is a loose subarc of $\alpha$ but not a subarc of $\alpha$ is precisely the case where there is a subarc relationship, in the type-A sense, between some arc in $\tilde\alpha'$ and some arc in $\tilde\alpha$.

If $\Theta$ is the restriction of a symmetric congruence $\Theta'$, then considering the symmetric arcs/pairs corresponding to $\alpha$ and $\alpha'$, closure of $U$ under passing to loose subarcs follows immediately from the observation in the previous paragraph and the fact that $\Theta'$ is closed under passing to subarcs in the type-A sense.

Conversely, suppose $U$ is closed under passing to loose subarcs.
Let $C$ be the complement of $U$ in the set of all type-B arcs.
Each type-B arc in $C$ corresponds to a symmetric arc or symmetric pair on $2n$ points.
Let $C'$ be the set of all of these type-A arcs on $2n$ points.
Let $\Theta'$ be the congruence on $S_{2n}$ generated by contracting the arcs in $C'$.
Because $C'$ is a symmetric set of arcs, $\Theta'$ is a symmetric congruence.
We will show that $\Theta'$ restricts to $\Theta$.

If $\Theta'$ fails to restrict to $\Theta$, then it restricts to a strictly coarser congruence $\tilde\Theta$.
In particular, there is some join-irreducible element contracted by $\tilde\Theta$ but not by $\Theta$, corresponding to a type-B arc $\beta$.
Let $\beta'$ be the symmetric arc associated to $\beta$ or one of the arcs in the symmetric pair associated to $\beta$.
Since $\beta$ is contracted by $\tilde\Theta$, which is symmetric, there is some arc $\alpha'\in C'$ that is a subarc of $\beta'$ in the type-A sense.
Let $\alpha$ be the type-B arc that arises from $\alpha'$ by passing to the orbifold model.
Since $\beta$ is not contracted by $\Theta$, $\alpha$ is not a subarc of $\beta$ in the type-B sense.  
The observation says that $\alpha$ is a loose subarc of $\beta$, and this contradiction implies that $\Theta'$ restricts to~$\Theta$.
\end{proof}

The following corollary is immediate.

\begin{corollary}\label{conA meet sub}
$\Con_A(B_n)$ induces a sublattice of the congruence lattice of $B_n$. 
\end{corollary}

%\marginN{Which of the examples we have done are from symmetric A? 
%\ra{yes! we know about 6.1 (no) and 6.2 (yes?), and we're thinking hard about the congruence in 6.3 (hard maybe, emily thinks no).} \re{Take a long arc with right endpoint at 1. This is contracted above. In type B, this arc has no  orbifold superarcs, but in type A (as a pair of symmetric arcs) it does! In fact, we can make a type A super arc which is orbifold and passes left of 1. For this reason, 6.3 is not symmetric.} \ra{This is definitely correct, and I feel fully convinced that 6.1 is NO and that 6.2 is YES, so I took a shot at writing it out in case we want it. It might not be good, but it's something :)}
%\rn{Say something about Cambrian congruences too.}
%\rn{The bipartite and linear biCambrian congruence is the restriction of a symmetric congruence.  This is presumably not hard to see from \cref{sym cong}, but Emily gave a good general argument for arbitrary biCambrian congruences. }
%}

To illustrate \cref{sym cong}, we consider which of the congruences in \cref{ex sec} are in $\Con_A(B_n)$.
The parabolic congruences of \cref{ex:B parabolic} are easily seen to be restrictions of parabolic congruences on $S_{2n}$.
Similarly, each Cambrian congruence (\cref{ex:B cambrian}) on $B_n$ is the restriction of a Cambrian congruence on $S_{2n}$.
Therefore also each biCambrian congruence (including those in \cref{ex:B bip bicambrian,ex:B lin bicambrian}) is in $\Con_A(B_n)$ because $\Con_A(B_n)$ induces a sublattice of $\Con(B_n)$.
The reader can also verify these statements using \cref{sym cong}.

\cref{sym cong} is the simplest way to determine whether congruences from \cref{ex:B diagram} are in $\Con_A(B_n)$.
The congruence described in \cref{simion finest} is not in $\Con_A(B_n)$ because it contracts all long arcs, some of which are loose subarcs of uncontracted orbifold arcs.
The congruence in \cref{nonhom finest} is in $\Con_A(B_n)$: the only uncontracted orbifold arc has upper endpoint 1 and the only uncontracted long arcs are those with lower endpoint 1, so no loose subarc relationship seen in \cref{fig:loose} may occur.
The congruence in \cref{delta finest} is not not in $\Con_A(B_n)$, since there exists a long arc with right endpoint at 1 which is contracted and is a loose subarc of an uncontracted orbifold arc passing left of 1.

% \section*{Critical reading record}
% On the following chart, let's record a ``good'' if we have read critically and are feeling happy with the (sub)section modulo any unresolved conversations in marginal notes.

% \noindent
% \begin{tabular}{c||c|c|c}
% Section                 &Ashley &Emily  &Nathan \\\hline\hline
% \ref{intro sec}	        & good  &   goodish    & good      \\  \hline  
% \ref{prelim sec}        & good  &   good    & good      \\  \hline  
% \ref{lat sec}           & good  &  good     & good      \\  \hline  
% \ref{weak sec}          & good  &   good    & good      \\  \hline  
% \ref{shard sec}	        & good  &  good     & good      \\   \hline 
% \ref{A sec}	            & good  & good      & good      \\  \hline  
% \ref{B sec} intro       & good  &  good     & good      \\    \hline
% \ref{B sym sec}         & good  &  good     & good      \\    \hline
% \ref{B orb sec}	        & good  &       &  good     \\    \hline
% % \ref{D shard sec}       & good  &       & good      \\    \hline
% % \ref{D nc sec}          & good  &       & good      \\    \hline
% % \ref{D arrow sec}       & good  &       & good      \\    \hline
% % \ref{D subarc sec}      & good  &       & good    \\   \hline 
% % \ref{D superarc sec}	& good  &       & good      \\    \hline
% \ref{ex sec}            &       &       &       
% \end{tabular}


\begin{thebibliography}{27}

% \bibitem{Asai}
% S. Asai,
% \textit{Bricks over preprojective algebras and join-irreducible elements in Coxeter groups.}
% J. Pure Appl. Algebra \textbf{226} (2022), no. 1, Paper No. 106812. 

\bibitem{Barnard}
E. Barnard,
\textit{The canonical join complex.}
Electron. J. Combin. \textbf{26} (2019), no. 1, Paper No. 1.24, 25 pp. 

\bibitem{bicat}
E.\ Barnard and N.\ Reading,
\textit{Coxeter-biCatalan combinatorics.}
J.\ Algebraic Combin.\ \textbf{47(2)} (2018), 241--300.

\bibitem{Darcs}
E.\ Barnard,  N.\ Reading, and A.\ M.\ Tharp,
\textit{Noncrossing arc diagrams of type D.}
In preparation, 2024.

%\bibitem{Bj-Br}
%A. Bj\"{o}rner and F. Brenti,
%\textit{Combinatorics of Coxeter groups.}
%Graduate Texts in Mathematics, \textbf{231},
%Springer, New York, 2005. 

\bibitem{BEZ}
A. Bj\"{o}rner, P. Edelman and G. Ziegler,
\textit{Hyperplane Arrangements with a Lattice of Regions.}
Discrete Comput. Geom. \textbf{5} (1990), 263--288.

%\bibitem{Bourbaki}
%N. Bourbaki,
%\textit{Lie groups and Lie algebras. Chapters 4--6.}
%Springer-Verlag, Berlin, 2002.

\bibitem{BCDK}
M. Bousquet-M\'{e}lou, A. Claesson, Anders, M. Dukes, and S. Kitaev,
\textit{$(2+2)$-free posets, ascent sequences and pattern avoiding permutations.}
J. Combin. Theory Ser. A \textbf{117} (2010), no. 7, 884--909. 

\bibitem{affncA}
L. Brestensky and N. Reading,
\textit{Noncrossing partitions of an annulus.}
Preprint, 2022. (\href{http://arxiv.org/abs/2212.14151}{\texttt{arXiv:2212.14151}}),
to appear in Comb. Theory.


\bibitem{CasPolMor}
N. Caspard, C. Le Conte de Poly-Barbut, and M.  Morvan, 
\textit{Cayley lattices of finite Coxeter groups are bounded.}
Adv. in Appl. Math. \textbf{33} (2004), no. 1, 71--94. 

\bibitem{FreeLattices}
R. Freese, J. Je\v{z}ek, and J. Nation,
\textit{Free lattices.} 
Mathematical Surveys and Monographs, \textbf{42}. 
American Mathematical Society, Providence, RI, 1995.

%\bibitem{Humphreys}
%J. Humphreys,
%\textit{Reflection Groups and Coxeter Groups}, Cambridge Studies in Advanced Mathematics, \textbf{29},
%Cambridge Univ. Press, 1990.

\bibitem{rectangle}
S. E. Law and N. Reading,
\textit{The Hopf algebra of diagonal rectangulations.}
J. Combin. Theory Ser. A \textbf{119} (2012), no. 3, 788--824. 

\bibitem{Poly-Barbut}
C. Le Conte de Poly-Barbut,
\textit{Sur les treillis de Coxeter finis.}
Math. Inform. Sci. Humaines No. \textbf{125} (1994), 41--57. 

% \bibitem{IRRT}
% Osamu Iyama, Nathan Reading, Idun Reiten, and Hugh Thomas,
% \textit{Lattice structure of Weyl groups via representation theory of preprojective algebras.}
% Compos. Math. \textbf{154} no. 6 (2018), 1269--1305.

\bibitem{shardpoly}
A. Padrol, V. Pilaud, and J. Ritter,
\textit{Shard Polytopes.}
International Mathematics Research Notices \textbf{2023}, Issue 9, 7686--7796, https://doi.org/10.1093/imrn/rnac042

% \bibitem{petersen}
% T. K. Petersen,
% \textit{On the shard intersection order of a Coxeter group.}
% SIAM J. Discrete Math. \textbf{27} (2013), no. 4, 1880--1912. 

\bibitem{quotientopes}
V. Pilaud and F. Santos,
\textit{Quotientopes.} 
Bull. Lond. Math. Soc. \textbf{51} (2019), no. 3, 406--420. 

\bibitem{hyperplane}
N. Reading,
\textit{Lattice and order properties of the poset of regions in a hyperplane arrangement},
Algebra Universalis, \textbf{50} (2003), 179--205.

\bibitem{congruence}
N. Reading,
\textit{Lattice congruences of the weak order},
Order \textbf{21} (2004) no.~4, 315--344.

\bibitem{con_app}
N.~Reading,
\textit{Lattice congruences, fans and Hopf algebras.}
J. Combin. Theory Ser. A {\bf 110} (2005) no.~2, 237--273.

\bibitem{cambrian}
N.\ Reading,
\textit{Cambrian Lattices.} 
Adv.\ Math.\ \textbf{205} (2006) no.\ 2, 313--353. 

%\bibitem{sortable}
%N.~Reading,
%\textit{Clusters, Coxeter-sortable elements and noncrossing partitions.}
%Trans. Amer. Math. Soc. \textbf{359} (2007), no. 12, 5931--5958.

%\bibitem{plane}
%N. Reading.
%\textit{Noncrossing partitions, clusters and the Coxeter plane.}
%S\'em. Lothar. Combin. \textbf{63} (2010) Art. B63b, 32 pages.

\bibitem{shardint}
N. Reading, 
\textit{Noncrossing partitions and the shard intersection order.} 
 J. Algebraic Combin. \textbf{33} (2011), no. 4, 483--530.

\bibitem{arcs}
N. Reading,
\textit{Noncrossing arc diagrams and canonical join representations.}
SIAM J. Discrete Math. \textbf{29} (2015), no. 2, 736--750. 

\bibitem{diagram}
N.\ Reading, 
\textit{Lattice homomorphisms between weak orders.} 
Electron.\ J.\ Combin.\ \textbf{26}, (2019), Article Number P2.23.

\bibitem{regions9} 
N. Reading, 
\textit{Lattice theory of the poset of regions.}
Lattice Theory: Special Topics and Applications, Volume \textbf{2}, ed. G. Gr\"{a}tzer and F. Wehrung, 
399--487,
Birkh\"{a}user/Springer, Cham, 2016. 

\bibitem{regions10} 
N. Reading, 
\textit{Finite Coxeter Groups and the weak order.}
Lattice Theory: Special Topics and Applications, Volume \textbf{2}, ed. G. Gr\"{a}tzer and F. Wehrung, 
489--561,
Birkh\"{a}user/Springer, Cham, 2016. 

\bibitem{surfnc}
N. Reading.
\textit{Noncrossing partitions of a marked surface.}
Preprint, 2022 (\href{http://arxiv.org/abs/2212.13799}{\texttt{arXiv:2212.13799}}),
to appear in SIAM J. Discrete Math.

\bibitem{ReadingTharp}
N.\ Reading and A.\ M.\ Tharp,
\textit{The shard intersection order on noncrossing arc diagrams.}
In preparation, 2025.

\bibitem{RitterThesis}
J.\ Ritter,
\textit{Shard polytopes and quotientopes for lattice congruences of the weak order.} 
Ph.D. Thesis, Institut Polytechnique de Paris, 2021.

\bibitem{SanWeh}
L. Santocanale and F. Wehrung,
\textit{Sublattices of associahedra and permutohedra.}
Adv. in Appl. Math. \textbf{51} (2013), no.~3, 419-445. 

\bibitem{TharpThesis}
A. M. Tharp,
\textit{Arcs and shards.}
Ph.D. Thesis, North Carolina State University, 2023.

\bibitem{Thomas}
H. Thomas,
\textit{Tamari lattices and noncrossing partitions in type B.}
Discrete Math.\ \textbf{306} (2006), no.~21, 2711--2723.


\end{thebibliography}
\end{document}